\documentclass[11pt]{amsart}
\usepackage{a4wide}
\usepackage{amssymb}
\usepackage{amsthm}
\usepackage{amsmath}
\usepackage{amscd}
\usepackage{mathrsfs}
\usepackage{bbm}
\usepackage[all]{xy}
\DeclareMathAlphabet{\mathpzc}{OT1}{pzc}{m}{it}
\usepackage{verbatim}

\numberwithin{equation}{section}

\theoremstyle{plain}
\newtheorem{theorem}{Theorem}[section]
\newtheorem{corollary}[theorem]{Corollary}
\newtheorem{lemma}[theorem]{Lemma}

\theoremstyle{definition}
\newtheorem{definition}[theorem]{Definition}
\newtheorem{remark}[theorem]{Remark}

\theoremstyle{remark}

\newcommand{\A}{\mathbb{A}}
\newcommand{\R}{\mathbb{R}}
\newcommand{\Q}{\mathbb{Q}}
\newcommand{\Z}{\mathbb{Z}}
\newcommand{\N}{\mathbb{N}}
\newcommand{\C}{\mathbb{C}}

\renewcommand{\H}{\mathbb{H}}

\newcommand{\leg}[2]{\left( \frac{#1}{#2} \right)}
\newcommand{\kzxz}[4]{\left(\begin{smallmatrix} #1 & #2 \\ #3 & #4\end{smallmatrix}\right) }

\newcommand{\im}{\operatorname{Im}}
\newcommand{\re}{\operatorname{Re}}

\newcommand{\odd}{\operatorname{oddity}}

\newcommand{\calD}{\mathcal{D}}

\newcommand{\calG}{\mathcal{G}}
\newcommand{\calH}{\mathcal{H}}

\newcommand{\calK}{\mathcal{K}}

\newcommand{\calM}{\mathcal{M}}

\newcommand{\calQ}{\mathcal{Q}}

\newcommand{\calS}{\mathcal{S}}
\newcommand{\calT}{\mathcal{T}}
\newcommand{\calU}{\mathcal{U}}

\newcommand{\calZ}{\mathcal{Z}}

\newcommand{\frake}{\mathfrak e}

\newcommand{\bs}{\backslash}

\newcommand{\SL}{\operatorname{SL}}
\newcommand{\GL}{\operatorname{GL}}

\newcommand{\Mp}{\operatorname{Mp}}

\newcommand{\Hom}{\operatorname{Hom}}

\newcommand{\sig}{\operatorname{sig}}

\newcommand{\End}{\operatorname{End}}

\newcommand{\ord}{\operatorname{ord}}
\newcommand{\id}{\operatorname{id}}

\begin{document}

\title[Standard $L$-function of a vector valued modular form]{The standard $L$-function attached to a vector valued modular form}
\author{Oliver Stein}
\address{Fakult\"at f\"ur Informatik und Mathematik\\ Ostbayerische Technische Hochschule Regensburg\\Galgenbergstrasse 32\\93053 Regensburg\\Germany}
\email{oliver.stein@oth-regensburg.de}

\begin{abstract}
We define two $L$-functions associated to a common vector valued eigenform  $f$ transforming with the ``finite'' Weil representation. The first one can be seen as a standard zeta function defined by the eigenvalues of $f$. The second one can be interpreted as standard $L$-function defined as an Euler product where each $p$-factor is a rational function in terms of two unramified characters of the $p$-adic field $\Q_p$. We show that both $L$-functions are related and prove further that they both can be continued meromorphically to the whole complex $s$-plane.  
\end{abstract}

\maketitle

\section{Introduction}
Vector valued modular forms transforming with the Weil representation   play a prominent role in the theory of Borcherds products, see e. g. \cite{Bo} or \cite{Br1}: The weakly holomorphic forms of this type serve as input to the celebrated Borcherds lift, which maps them to meromorphic modular forms on orthogonal groups whose zeroes and poles are supported on special divisors and which possess an infinite product expansion. This lift has many important applications in geometry, algebra and in the theory of Lie algebras. Since their appearance in the works of Borcherds and Bruinier, a lot of research  regarding this type of modular forms has been done and most of the classical theory of modular forms has been established over the past years (see e. g. \cite{Br1}, \cite{BS}, \cite{Br2}, \cite{St1} or \cite{Mu1} and \cite{St3}).
However, still not much is known about associated Dirichlet series. In \cite{BS}, a zeta function of the form $\sum_{\substack{d\in \N\\(d,N)=1}}\lambda_f\kzxz{d^2}{0}{0}{1}d^{-s}$ was introduced. In \cite{St1} the analytic properties of the slightly different zeta function \eqref{eq:zeta_func_eigenvalues} were investigated. So far, to the best of my knowledge, there has been no standard $L$-function specified for this type of modular forms.
So, it is quite natural to define such an $L$-function and to study its analytic properties, which is the main goal of the present paper. 
This paper can be understood as a second part of \cite{St3}. Therein, we laid the foundations to define a standard $L$-function by developing a  theory of vector valued automorphic forms 
corresponding to  vector valued modular forms transforming according to the finite Weil representation.
This includes the study of the structure of a local vector valued spherical Hecke algebra  depending on the Weil representation $\omega_f$. All structural statements regarding this algebra assume that the local discriminant group is {\it anisotropic} and its order is {\it odd}. Consequently, the definition of our standard $L$-function relies also on this restriction. \newline
 In the present paper, two  types  of $L$-functions are defined. Both are associated to a common Hecke eigenform $f$. The first one is  defined in terms of the eigenvalues $\lambda_f(D)$ of a family of  Hecke operators $T(D)$ and is called standard zeta function in this paper. 
 The second one, the standard $L$-function, is constructed as an Euler product in terms of the Satake parameters $\alpha_{i,p}$ of $f$ (which are ultimately given by two unramified characters of $\Q_p$), where $p$ is a prime).  It can be written in the form $L(s,f) = \prod_{p<\infty}L_p(s,f)$ with
\begin{equation}\label{eq:local_L_func_rational_expr}
L_p(s,f) = R(\alpha_{i,p}, p^s),
\end{equation}
where $R(\alpha_{i,p},p^s)$ is a rational expression depending on the parameters $\alpha_{i,p}$ and the prime power $p^{s}$.
For both types of $L$-functions it is shown that they can be continued meromorphically to the whole complex $s$-plane. We also prove that both $L$-functions are related by an explicit equation. 

Note that some results of this paper  may already be covered in a paper of  Yamana (\cite{Ya}), who established analytic properties of a standard $L$-function associated to an automorphic representation of the metaplectic group $\Mp(2n)$. Yet, these results seem not to be immediately applicable to our setting and it is unclear how to recover the computations of the local Euler factors from them (which play an important role in applications - see the remarks below).

The  paper at hand is intended as a first step towards a more comprehensive study of $L$-functions associated to vector valued automorphic forms for the Weil representation:  Scheithauer (see \cite{Sch} for example) and others investigated extensively a lifting from scalar valued modular forms for $\Gamma_0(N)$, $N$ the level of the lattice $L$ (see below for more details), to vector valued modular forms transforming with the Weil representation, which commutes with Hecke operators on both sides. On the other hand, it is well known that there is a well established theory of automorphic forms  and automorphic representations of $\GL(2)$ connected to modular forms for $\Gamma_0(N)$ (see e. g. \cite{Ge}). It would be interesting to compare a vector valued automorphic form obtained from a lifted scalar valued modular form with the corresponding scalar valued automorphic form. I expect that there is a relation between the associated standard $L$-functions on both sides. This would hopefully allow one to place the standard $L$-functions of this paper within in the Langlands framework. 
Another aspect of vector valued modular forms is their relation to Jacobi forms of lattice index (see \cite{Wa}, p. 2085, \cite{Br1}, Ex. 1.3 or \cite{BS}, Remark 4.11). I think it is worthwhile to investigate whether this relation carries over to the attached standard $L$-functions on both sides with the same goal in mind as before. 
I hope to come back these questions in the near future.\newline
Also, this paper (and \cite{St1})  can be used to prove   more general results on the injectivity of the Kudla-Millson lift along the lines of \cite{BF}. The corresponding proof relies on the computation of the local Euler factors of the introduced standard $L$-function.  In turn, a generalization of a converse theorem for lattices of level $p$ as stated in \cite{Br2} can be proven. This is indeed achieved in \cite{St2}. Although our results in the present paper are restricted to anisotropic discriminant forms, based on them, we have established one the most general converse theorems for the Borcherds lift (to the best of my knowledge).  

Let us describe the content of the paper in more detail. To this end, let $(L,(\cdot,\cdot))$ be an even lattice of even rank $m$ and type $(b^+, b^-)$ with (even) signature $\sig(L)=b^+-b^-$ and level $N$. Throughout this paper we assume that $N$ is {\it odd}. Associated to the bilinear form $(\cdot,\cdot)$ there is a quadratic form $q$. The modulo 1 reduction of $(\cdot, \cdot)$ and $q$  defines a bilinear form and quadratic form, respectively, on the discriminant form $D=L'/L$. Here $L'$ is the dual lattice of $L$.
The Weil representation $\rho_L$ is a representation of $\Gamma = \SL_2(\Z)$ on the group ring $\C[D]$.
In the first part of the paper a certain standard zeta function $\calZ(s,f)$  attached to common eigenform $f$ is studied. Its  local $p$-part is given by the eigenvalues of the Hecke operators $T\kzxz{p^{-k}}{0}{0}{p^{-l}}$ with $0\le k\le l$ and $k+l\in 2\Z$:
\[
\calZ_p(s,f)=\sum_{(k,l)\in \Lambda_+}\lambda_{f}\kzxz{p^{-k}}{0}{0}{p^{-l}}p^{-s(k+l)},
\]
where $\Lambda_+$ is specified in \eqref{eq:lambda_+}.
A connection to the local standard zeta function in \cite{St1} is then proved and thereby  analytic properties of $\calZ(s,f)$ can be deduced from those of the zeta function in \cite{St1}.

The second part of this article deals with definition of the standard $L$-function of an eigenform $f$ and subsequently investigates its analytic properties. To sketch the central ideas, we briefly summarize some facts and notation from \cite{St3}: 
As usual, denote with  $\Z_p$ the ring of $p$-adic integers and let $\widehat{\Z}$ be  $\prod_{p<\infty}\Z_p$. Moreover, by $\omega_f=\bigotimes_{p<\infty}\omega_p$ we mean the Weil representation of $\SL_2(\widehat{\Z})$  on a space $S_L$ (isomorphic to $\C[D]$). In \cite{St3}, Section 3, it is explained how $\omega_f$ can be extended to some subgroup of $\GL_2(\A_f)$, where $\A_f$ denotes the finite adeles. Let $\calQ_p$ be a subgroup of $\GL_2(\Q_p)$ and $\calK_p$ a subgroup of $\GL_2(\Z_p)$. In terms of these groups and the local Weil representation $\omega_p$ one can define a vector valued spherical Hecke algebra $\calH(\calQ_p//\calK_p,\omega_p)$.  In \cite{St3} its structure is determined and a set of generators  $\{T_{k,l}\; |\; (k,l)\in \Lambda_+\}$ is specified. Similar to other types of automorphic forms, there is an action of $\calH(\calQ_p//\calK_p,\omega_p)$  on the space of vector valued automorphic forms based on convolution. Any eigenform $F$ of all $T_{k,l}$ with eigenvalues $\lambda_{F,p}(T_{k,l})$ defines a $\C$-algebra homomorphism, which in turn determines a pair of unramified characters $(\chi_{F,p}^{(1)}, \chi_{F,p}^{(2)})$. \newline
A common way described in the literature (see for instance \cite{Bo2}, \cite{BM}, \cite{BoSP} or \cite{Sh1}) to obtain \eqref{eq:local_L_func_rational_expr}, is to factorize the Hecke series
\[
\sum_{(k,l)\in \Lambda_+}\lambda_{F,p}(T_{k,l})p^{-s(k+l)}. 
\]
We follow Arakawa \cite{Ar}, Section 5, to establish this factorization by means of an integral representation of the Hecke series. On the one hand, we have up to some constant
\[
\sum_{(k,l)\in \Lambda_+}\lambda_{F,p}(T_{k,l})p^{-s(k+l)} = \int_{\calQ_p}\langle \nu_s(g),\varphi_p^{(0)}\rangle\langle\phi_\chi(g),\varphi_p^{(0)}\rangle dg,
\]
where $\chi=(\chi_1,\chi_2)$ is a pair of unramified characters of $\Q_p$ determined by the eigenvalues $\lambda_{F,p}(T_{k,l})$. 
On the other hand, the integral on the right-hand side of this equation can be evaluated resulting in the identity
\[
\int_{\calQ_p}\langle \nu_s(g),\varphi_p^{(0)}\rangle\langle\phi_\chi(g),\varphi_p^{(0)}\rangle dg = \frac{1+\chi_1(p)\chi_2(p)p^{-2s+1}}{(1-\chi_1(p)^2p^{-2s+1})(1-\chi_2(p)p^{-2s+1})}
  \]
  if $p$ is coprime to $|D|$ and
  \[
  \kappa_p\frac{1-\chi_1(p)^2\chi_2(p)^2p^{-4s+2}}{(1-\chi_1(p)\chi_2(p)p^{-2s})(1-\chi_1(p)^2p^{-2s+1})(1-\chi_2(p)^2p^{-2s+1}))}
  \]
  if $p$ divides $|D$ and the $p$-order of the $p$-group of $D$ is odd. The formula is slightly more complicated if $\ord_p(D_p)\in 2\Z$ (see Lemma \ref{lem:int_repr_series}).
  The  notation for the two formulas above can be found in Section \ref{subsec:standard_L_func}.
  As indicated above, the right-hand side of the last  identity gives rise to our definition of the local standard $L$-function $L_p(s,F)$.

  The last part of the paper deals with the analytic properties of $L(s,F)$. Following \cite{Ar}, Section 5, again, we find a direct relation between the standard $L$-function and the zeta function $\calZ(s,f)$.
  The property of being meromorphic carries over to $L(s,F)$. As $\calZ(s,f)$ does not satisfy a functional equation, it is unclear whether $L(s,F)$ obeys a functional equation.  

\section{Notation}\label{sec:notation}
We adopt the notation in \cite{St3}. Nevertheless, for convenience of the reader, we introduce the most common symbols in the paper. As usual, we let $e(z)$, $z\in \C$, be the abbreviation for $e^{2\pi i z}$. For any prime $p\in \Z$ by $\Q_p$ we mean the field of $p$-adic numbers and by $\Z_p$ its ring of $p$-adic integers. 

The groups listed below will appear frequently in this paper.
\begin{equation}\label{def:p_adic_subgroups}
  \begin{split}
    & \calG(R) = \{(M, r)\in \GL_2(R)\; |\; \det(M) = r^2\} \text{ for any commutative ring $R$ with $1$},\\
    & N(\Q_p) = \{ \kzxz{1}{r}{0}{1}\; |\; r\in \Q_p\} \text{ and } N(\Z_p) \text{ accordingly},\\ 
    &\calQ_p=\{M\in \GL_2(\Q_p)\; |\; \det(M)\in (\Q_p^\times)^2\}, \\
    & \calK_p=\{M\in \GL_2(\Z_p)\; |\; \det(M)\in (\Z_p^\times)^2\}, \\
    &\calM_p=\{M\in \kzxz{r_1}{0}{0}{r_2}\in \GL_2(\Q_p)\; |\; \det(M)\in (\Q_p^\times)^2\}, \\
    &\calD_p=\calM_p\cap \calK_p  \\
    \end{split}
\end{equation}
 Moreover, the subgroup 
  \begin{equation}\label{eq:K_0}
    \calK_0(p)=\left\{\kzxz{a}{b}{c}{d}\in \calK_p\; |\; c\equiv 0\bmod{p}\right\}
  \end{equation}
  of $\calK_p$ will be relevant.
Throughout the paper we use the following abbreviations for certain elements of these groups
\begin{align*}
  n\_(c) = \kzxz{1}{0}{c}{1}, \quad n(b) = \kzxz{1}{b}{0}{1}, \quad m(s) = \kzxz{s}{0}{0}{s^{-1}}, \quad m(t_1,t_2) = \kzxz{t_1}{0}{0}{t_2}\text{ and } w=\kzxz{0}{1}{-1}{0}. 
\end{align*}
  Since $\calQ_p$ is locally compact (see \cite{St3}, Lemma 4.3), we may fix a Haar measure on $\calQ_p$ such that $\int_{G\cap \calK_p} dg = 1$ for any of the groups $G=\calQ_p, \calM_p$ and $N(\Q_p)$. We denote with $\mu(K)$ the measure of any subgroup $K$ of $\calK_p$. 
  
Moreover, we will make frequently use of the following subsets of $\Z^2$:
\begin{equation}\label{eq:lambda_+}
  \begin{split}
    &\Lambda = \{(k,l)\in \Z^2\; |\; k\le l \text{ and } k+l\in 2\Z\} \text{ and }\\
    & \Lambda_+ = \{(k,l)\in \Lambda\; |\; k, l\ge 0 \}.
    \end{split}
  \end{equation}
Finally, as usual, we write $\H=\{\tau \in \C\; |\; \im(\tau) > 0\}$ for the complex upper half plane and $\leg{\cdot}{d}$ for the Legendre symbol.  

\section{Review of \cite{St3}}
In this section we would like to present some results of \cite{St3} which we will be relevant for the present paper. This will hopefully improve the readability of the present work. 
\subsection{Local Weil representation $\omega_p$, \cite{St3}, Sect. 3}\label{subsec:local_weil}
Let $(L, (\cdot,\cdot))$ be an even, non-degenerated lattice of signature $b^+-b^-$ and even rank $b^++b^-$. Associated to $L$ is dual lattice $L'$ and the quadratic module $D=L'/L$ with an odd group order. Moreover, put $V=L\otimes \Q$. Let $p$ be an odd  prime dividing $|D|$. Then $D_p$ means the $p$-group of $D$ and $L_p=L\otimes \Z_p$ is the corresponding $p$-adic lattice. Note that  $D_p$ is isomorphic to the group $L_p'/L_p$.  By $\omega_p$ we mean the ``$p$''-part of the adelic  Weil representation $\omega_f$ of $\SL_2(\A)\times O(V)(\A)$, where $\A$ is the ring of the adeles. In Section 3 of \cite{St3} the extension of the local Weil representation $\omega_p$ from $\SL_2(\Z_p)$  to the group $\calK_p$ is specified. Strictly speaking, the extensions of $\omega_p$ to $\calK_p$ and of $\omega_f$ to $\calK$ have to be understood on the corresponding lifted groups  ($\calK(p)$ in \cite{St3}, (3.12), and $\calK' =\prod_{p<\infty}\calK(p)$) incorporating the choice of the square root factor. As in \cite{St3}, we suppress this additional datum from the notation whenever no confusion can arise. Since several formulas later in the present paper depend on the explicit evaluation of $\omega_p$ on $\calK_p$, we briefly recapitulate the relevant facts on $\omega_p$. The representation space of $\omega_p$ is $S_{L_p}=\bigoplus_{\mu_p\in D_p}\C\varphi_p^{(\mu_p)}$, where $\varphi_p^{(\mu_p)}$ is a Schwartz-Bruhat function given by $\mathbbm{1}_{\mu_p+L_p}$, the characteristic function of the coset $\mu_p+L_p$ in $L_p'/L_p$. 
Let
\[
\psi_p:\Q_p/\Z_p\rightarrow \C^\times, x_p\mapsto \psi_p(x_p) = e(x_p')
\]
the conjugate of the ``standard additive'' character of $\Q_p$, where   $x_p'\in \Q/\Z$ is the principal part of $x_p$. Associated to $\psi_p$, in the Schr\"odinger model, $\omega_p$ is defined on the generators of $\SL_2(\Z_p)$ as follows:
\begin{equation}\label{eq:weil_rep_explicit}
  \begin{split}
    &    \omega_p(n(b))\varphi_p^{(\mu)} = \psi_p(bq(\mu))\varphi_p^{(\mu)} \\
    & \omega_p(w)\varphi_p^{(\mu)} = \frac{\gamma_p(D_p)}{|D_p|^{1/2}}\sum_{\nu_p\in D_p}\psi_p((\mu_p,\nu_p))\varphi_p^{(\nu_p)}\\
    & \omega_p(m(a))\varphi_p^{(\mu_p)} = \chi_{V,p}(a)\varphi_p^{(a^{-1}\mu_p)},
    \end{split}
\end{equation}
where $\gamma(D_p)$ is the local Weil index and  $\chi_{V,p}(a) = (a, (-1)^{m/2}|D_p|)_p = \leg{a}{|D_p|}$ is the local Hilbert symbol. 
We also repeat the action of diagonal matrices $m(t_1,t_2)\in \calK_p$  in the Weil representation $\omega_p$. We have 
\begin{equation}
    \omega_p(m(t_1,t_2))\varphi_p^{(\mu_p)} =\leg{t_1}{|D_p|}\varphi_p^{(t^{-1}t_2\mu_p)},
  \end{equation}
  where $\det(m(t_1,t_2)) = t^2\in \Z_p^\times$.
  If $p$ is coprime to $|D|$, then $L_p'/L_p = 0+L_p$ and $\omega_p$ is trivial.
  
  \subsection{The Hecke algebras $\calH(\calQ_p//\calK_p, \omega_p)$ (\cite{St3}, Sect. 4)}
  Let $p$ be an odd prime  and  $\calH(\calQ_p//\calK_p, \omega_p)$ be set of maps $F:\calQ_p\rightarrow \End(S_{L_p})$ satisfying
  \begin{enumerate}
  \item[i)]
    $F$ is supported on finitely many double cosets $\calK_pg\calK_p, \; g\in \calQ_p$.
  \item[ii)]
    $F(k_1gk_2) = \omega_p(k_1)\circ F(g)\circ \omega_p(k_2)$ for all $k_1,k_2\in \calK_p$ and $g\in \calQ_p$.
  \end{enumerate}
  $\calH(\calQ_p//\calK_p,\omega_p)$ becomes with respect to convolution an associative $\C$-algebra. In \cite{St3} a set of generators of $\calH(\calQ_p//\calK_p,\omega_p)$ is given. For a prime $p$ dividing $|D|$ a subalgebra $\calH^+(\calQ_p//\calK_p,\omega_p)$ is considered. Under the assumption that $D_p$ is an anisotropic quadratic module, a set of generators of $\calH^+(\calQ_p//\calK_p,\omega_p)$ is given by
  \begin{enumerate}
    \item
      \begin{equation}\label{def:T_k_l}
      T_{k,l}:S_{L_p}\rightarrow S_{L_p},\quad \varphi_p^{(\mu_p)}\mapsto \varphi_p^{p^{\frac{1}{2}(l-k)}\mu_p} = \varphi_p^{(0)}
      \end{equation}
      for $(k,l)\in \Lambda$ with $k<l$, where $T_{k,l}$ is supported on $\calK_pm(p^k,p^l)\calK_p$.
    \item
      \begin{equation}\label{def:T_k}
        T_k: S_{L_p}\rightarrow S_{L_p},\quad \varphi_p^{(\mu_p)}\mapsto \varphi_p^{(\mu_p)}
      \end{equation}
      for $(k,l)\in \Lambda$ with $k=l$, where $T_k$ is supported on $\calK_pm(p^k,p^k)\calK_p$.
  \end{enumerate}
  For $(p,|D|)=1$ the algebra $\calH(\calQ_p//\calK_p, \omega_p)$ is much simpler and well known. A set of generators given by
  \begin{equation}\label{def:T_k_l_coprime}
    \{\mathbbm{1}_{\calK_pm(p^k,p^l)\calK_p}\id_{S_{L_p}}\;|\; (k,l)\in \Lambda\},
  \end{equation}
  where $\mathbbm{1}_{\calK_pm(p^k,p^l)\calK_p}$ is the characteristic function of $\calK_pm(p^k,p^l)\calK_p$.
  The so called spherical map or Satake map is a means to get a better  understanding of  the structure of $\calH^+(\calQ_p//\calK_p,\omega_p)$ and $\calH(\calQ_p//\calK_p,\omega_p)$.
  If $p$ divides $|D|$, we consider a variant of the classical Satake map:
  \begin{equation}\label{def:satake_map}
  \begin{split}
    & \calS: \calH(\calQ_p//\calK_p, \omega_p) \rightarrow \calH(\calM_p//\calD_p,\omega_p{{{_{|S_{L_p}^{N(\Z_p)}}}}} ), \\
  &  T\mapsto \left(m\mapsto \delta(m)^{1/2}\sum_{n\in N(\Q_p)/N(\Z_p)}T(mn)_{|S_{L_p}^{N(\Z_p)}}\right).
  \end{split}
  \end{equation}
  Restricted to $\calH^+(\calQ_p//\calK_p,\omega_p)$, the map $\calS$ defines an algebra isomorphism between $\calH^+(\calQ_p//\calK_p,\omega_p)$ and $\calH^W(\calM_p//\calD_p, \omega_{p_{|S_{L_p}^{N(\Z_p)}}})$. 
  In the case of $(p,|D|)=1$, the composition of $T\mapsto \langle T,\varphi_p^{(0)}\rangle$ with the classical Satake map $S$
  \begin{equation}\label{def:satake_map_coprime}
     m\mapsto  \left(m\mapsto \delta(m)^{1/2}\int_{N(\Q_p)}f(mn)dn\right)
  \end{equation}
  gives an algebra isomorphism  from $\calH(\Q_p//\calK_p,\omega_p)$ to $\calH(\calM_p//\calD_p)^W$.  

  \subsection{Vector valued automorphic forms and the action of $\calH(\calQ_p//\calK_p,\omega_p)$, Sect. 5}
  Let $f\in S_\kappa(\rho_L)$ be a cusp form (for the definition of $S_\kappa(\rho_L)$ see the next section). We associate to $f$  a  vector valued automorphic form by
  \[
  F_f: \calG(\Q)\setminus \widetilde{\calG}(\A)\rightarrow S_L, \quad g\mapsto F_f((g)= \omega_f(k)^{-1}j(g_\infty,i)^{-\kappa}f(g_\infty i),
  \]
  where $g=\gamma(g_\infty \times k)$ describes the decomposition of $g$ with respect to the strong approximation theorem and the space $S_L$ is  defined as $S_{L_p}$ in Subsection \ref{subsec:local_weil}. Here $\widetilde{G}(\A)$ is a lifted group from $\calG(\A)$ to incorporate the the square root factor necessary for the Weil represenation, which is specified in detail in \cite{St3}, (5.4). In this work we follow the convention of \cite{St3} and omit the $\;\widetilde{}\;$ and the square root $\tilde{\calG}(\A)$ whenever no confustion can arise.
  The mapping $f\mapsto F_f$ is in fact an isomorphism from $S_\kappa(\rho_L)$ to the space $A_\kappa(\omega_f)$ of automorphic forms, which is defined similar to the classical space of elliptic automorphic forms. This isomorphism is denoted with $\mathscr{A}$.
  There is an action of the Hecke algebras $\calH^+(\calQ_p//\calK_p,\omega_p)$  (or $\calH(\calQ_p//\calK_p,\omega_p)$ if $(p,|D|)=1$) by convolution (analogous to the case of scalar valued spherical Hecke algebras) on automorphic forms:
  Let $F(g) = \bigotimes_{p<\infty} F_p(g)\in A_\kappa(\omega_f)$ and $T_p\in \calH^+(\calQ_p//\calK_p,\omega_p)$ (or $\calH(\calQ_p//\calK_p,\omega_p)$ if $(p,|D|)=1$)). Then we define 
   \begin{equation}\label{eq:global_adelic_hecke_op} 
\calT^{T_p}:A_{\kappa}(\omega_f)\rightarrow A_{\kappa}(\omega_f),\quad \calT^{T_p}(F)(g)=\sum_{x_p\in \calQ_p/\calK_p}R^{T_p}(\iota_p(x_p))F(g\iota_p(x_p)),
   \end{equation}
   where
   \begin{equation}\label{eq:shift_op}
    R^{T_p}(h) \text{ is the operator } \quad F\mapsto R^{T_p}(h)F= \bigotimes_{q<\infty}R_q^{T_p}(h_q)F_q 
  \end{equation}
  with
  \begin{equation}\label{eq:local_hecke_op}
    R_q^{T_p}(h_q)F_q(g_q) = 
    \begin{cases}
      F_q(g_\infty, g_qh_q), & q\not=p\\
      T_p(h_p)(F_p(g_\infty,g_p)), & q= p.
\end{cases}
  \end{equation}
  The action of these Hecke algebras via \eqref{eq:global_adelic_hecke_op} is compatible with that of Hecke operators on the space $S_\kappa(\rho_L)$. In fact, it can be proved that
  \begin{equation}\label{eq;compatibility_notcoprime}
    \calT^{T_{k,l}}(F_f) = F_{p^{(k+l)(\frac{\kappa}{2}-1)}T(m(p^{-k},p^{-l}))f}
  \end{equation}
 for all $(k,l)\in \Lambda$, which is an analogue of the classical result for scalar valued modular forms. Based on \eqref{eq;compatibility_notcoprime}, it is important to observe that a common eigenform $f\in S_\kappa(\rho_L)$  of the Hecke operators $T(m(p^{-k},p^{-l}))$ for all $(k,l)\in \Lambda$ determines a via its eigenvalues a $\C$-algebra homomorphism of $\calH^+(\Q_p//\calK_p,\omega_p)$ (or $\calH(\Q_p//\calK_p,\omega_p)$) (see \cite{St3}, Remark 5.10. for more details).  
  
\section{Hecke operators $T(m(p^k,p^l))$ and standard zeta-functions}\label{subsec:hecke_operators_vec_val}
In this section we briefly summarize some facts on lattices, discriminant forms, Gauss sums  and the ``finite'' Weil representation. We also recall the definition of vector valued modular forms for the Weil representation and some related theory relevant for the present paper. Subsequently, we explain how to extend the definition of the Hecke operator $T(p^{2l})^*$ in Definition 5.5 of \cite{BS} to the action of double cosets of the form $\Gamma m(p^k,p^l)\Gamma$. In \cite{BS} a standard zeta function associated to an eigenform of all Hecke operators $T(m(n^2,1)),\; n\in \N$ is introduced. Here, we consider a variant of this zeta function and state some analytic properties based on the results in \cite{St1}.

Let $L$ be a lattice of rank $m$  equipped with a symmetric $\Z$-valued bilinear form $(\cdot,\cdot)$ such that the associated quadratic form
\[
q(x):=\frac{1}{2}(x,x)
\]
takes values in $\Z$ for $ x\in L$. We assume that $m$ is even, $L$ is non-degenerate and denote its type by $(b^+,b^-)$ and its signature $b^+-b^-$ by $\sig(L)$. Note that $\sig(L)$ is also even. We stick with these assumptions on $L$ for the rest of this paper unless we state it otherwise. Further, let $L'$
be the dual lattice of $L$. 
Since $L\subset L'$, the elementary divisor theorem implies that $L'/L$ is a finite group. We denote  this group by $D$ and for any prime $p$ dividing $|D|$ by $D_p$ its $p$-group. The modulo 1 reduction of both, the bilinear form $(\cdot, \cdot)$ and the associated quadratic form, defines a $\Q/\Z$-valued bilinear form $(\cdot,\cdot)$ with corresponding $\Q/\Z$-valued quadratic form on $D$. We call $D$ combined with $(\cdot,\cdot)$ a discriminant form or a quadratic module. We call it anisotropic, if $q(\mu) = 0$ holds only for $\mu=0$. Further, we denote by $N$ the level of the lattice $L$. It is the smallest positive integer such that $Nq(\lambda)\in \Z$ for all $\lambda\in L'$.

Let $d$ an integer. By $g_d(D)$ we denote the Gauss sum
  \begin{equation}\label{eq:gauss_sum_d}
    g_d(D)=\sum_{\lambda\in D}e(dq(\lambda))
  \end{equation}
  and $g(D) = g_1(D)$. 
Since fractions of these Gauss sums are of some relevance in this paper, we gather some facts on the sums $g_d(D)$ and quotients thereof (see Lemma 2.1 in \cite{St3})
\begin{lemma}\label{lem:gauss_sum_prop}
  \begin{enumerate}
    \item[i)]
  The Gauss sums $g_d(D)$ satisfy the properties
  \begin{align*}
   & g_{-d}(D) = \overline{g_d(D)} \\
    & g_d(D\oplus D') = g_d(D)g_d(D') \\
    & g_{dr}(D) = g_d(D),
  \end{align*}
  where $r\in \Z$ is square in $(\Z/N\Z)^\times$. 
\item[ii)]
  If $d$ is coprime to $|D|$, we  have
\begin{equation}\label{eq:frac_gauss_sum}
  \frac{g(D)}{g_d(D)} =\leg{d}{|D|}e\left(\frac{(d-1)\odd(D)}{8}\right)
\end{equation}
If $|D|$ is odd, the right-hand side of \eqref{eq:frac_gauss_sum} simplifies to
the quadratic character
\begin{equation}\label{eq:frac_gauss_char}
  \chi_{D}(d) = \leg{d}{|D|}.
\end{equation}
\end{enumerate}
\end{lemma}


The Weil ``finite'' representation $\rho_L$ is a representation of $\Gamma = \SL_2(\Z)$ on the group ring $\C[D]$. We denote the standard basis of $\C[D]$  by $\{\frake_{\lambda}\}_{\lambda\in D}$. We refer to \cite{Br1} or \cite{BS} for the definition of $\rho_L$ on the generators of $\Gamma$.   

A holomorphic function $f: \H\rightarrow \C[D]$ is called a vector valued modular
form of weight $ \kappa $ and type $ \rho_L $ for $ \Gamma $ if
$ f\mid_{\kappa,L}\gamma= f $ for all $\gamma\in \Gamma$,
and if $f$ is holomorphic at the cusp $\infty$. Here
\begin{equation*}
f\mid_{\kappa,L}\gamma =
j(\gamma,\tau)^{-\kappa}\rho_L(\gamma)^{-1}f(\gamma\tau),
\end{equation*}
where
\[
j(\gamma,\tau) = \det(\gamma)^{-1/2}(c\tau+d)
\]
is the usual automorphy factor if $\gamma=\kzxz{a}{b}{c}{d}\in \GL_2^+(\R)$ and 
the last condition means that all Fourier coefficients $c(\lambda, n)$  of $f$ with $n <  0$ vanish. If in addition $c(\lambda,n) = 0$ for all $n = 0$, we call the corresponding modular form a cusp form. We denote by $M_\kappa(\rho_L)$ the
space of all such modular forms, by $S_\kappa(\rho_L)$ the subspace of  cusp forms. For
more details see e.g. \cite{Br1} or \cite{BS}. Note that $M_\kappa(\rho_L) = \{0\}$ unless
\begin{equation}\label{eq:sig_weight}
2\kappa\equiv \sig(L)\pmod{2}.
\end{equation}
Therefore, if the signature of $L$ is even, only non-trivial spaces of integral
weight can occur. 


In \cite{BS} Hecke operators $T(m(d^2,1),d)$ on $S_{\kappa}(\rho_L)$ were
defined in the usual way by the action of double cosets $\Gamma m(d^2,1)\Gamma$:
\begin{equation}\label{eq:hecke_op_copr}
f\mid_{\kappa,L} T(m(d^2,1),d)= \det(m(d^2,1))^{\kappa/2-1}\sum_{M\in \Gamma\bs \Gamma m(d^2,1) \Gamma} f\mid_{\kappa,L} (M,d),
\end{equation}
where $d$ is one of the square roots of $\det(m(d^2,1))$. More generally, this definition works for  Hecke operators $T(g,r)$ with $g$ being an element of some subgroup of $\GL_2^+(\Q)$   whose determinant is coprime to $N$ and satisfies $\det(g)\equiv r^2\bmod{N}$. As demonstrated in \cite{BS}, Chapter 5, \eqref{eq:hecke_op_copr}  defines still a Hecke operator if we assume $(d,N)>1$. Yet in general \eqref{eq:hecke_op_copr} does not make sense for matrices $g\in \GL_2^+(\Q)$ with $(\det(g),N)>1$ since there is no definition of $\rho_L$ for these matrices. However, in \cite{St1} a definition of the Weil representation for certain diagonal matrices in $\GL_2^+(\Q)$ was given. Here, we pick up these ideas and adapt them slightly to get smoother formulas and avoid several technical difficulties later in the paper:

Let $k, l\in \Z$ with $k|l$ and $(kl,N)> 1$. We assume further that $\frac{l}{k}$ is a square.
Obviously,
\begin{equation}\label{eq:diagonal_matrix_decomp}
  m(l,k) = m(k,k)m(\frac{l}{k},1) \text{ and } m(k^{-1},l^{-1}) = m(\frac{l}{k},1)m(l,l)^{-1}.
  \end{equation}
The action of $\rho_L(m(k,k))$ was given in \cite{BS}, (3.5), for alle integers $k$ with $(k,N)=1$, simply by  multiplication with $\displaystyle \frac{g(D)}{g_k(D)}$.
Now let $k$ be an integer with $(k,N)= \prod_ip_i^{e_i}$, $D_{p_i}$ the $p_i$-group of $D$ and $D(k)=\left(\bigoplus_iD_{p_i}\right)^\perp$ the orthogonal complement of $\bigoplus_iD_{p_i}$ in $D$. Then by construction, $(k, |D(k)|) = 1$ and in particular $k$ is coprime to the level of $|D(k)|$. 
Then we put
\begin{equation}\label{eq:weil_scalar_ncoprime}
\rho_L(m(k,k))\frake_\lambda = \frac{g(D(k))}{g_k(D(k))}\frake_\lambda
\end{equation}
and fittingly
\begin{equation}\label{eq;weil_scalar_ncoprime_inverse}
  \rho_L(m(k,k)^{-1})\frake_\lambda =\rho_L^{-1}(m(k,k))\frake_\lambda = \frac{g_k(D(k))}{g(D(k))}\frake_\lambda. 
  \end{equation}
Note that this definition is compatible with the one in \cite{BS} in the case that $(k,N)=1$. Also, for a prime $p$ dividing $|D|$ we have
\[
\rho_L(m(p^k,p^k))\frake_\lambda = \frac{g(D_p^\perp)}{g_{p^k}(D_p^\perp)}\frake_\lambda.
\]
Based on this and the decompositions \eqref{eq:diagonal_matrix_decomp}, 
\begin{equation}\label{eq:weil_diageonal_matrix_1}
\rho_L^{-1}(m(l,k))\frake_\lambda = \frac{g_k(D(k))}{g(D(k))}\rho_L^{-1}(m(\frac{l}{k},1))\frake_{\lambda}
\end{equation}
and
\begin{equation}\label{eq:weil_diageonal_matrix}
  \rho_L^{-1}(m(k^{-1},l^{-1}))\frake_\lambda = \frac{g(D(l))}{g_l(D(l))}\rho_L^{-1}(m(\frac{l}{k},1))\frake_\lambda.
\end{equation}
Here, the Weil representations on the right-hand side of \eqref{eq:weil_diageonal_matrix_1} and \eqref{eq:weil_diageonal_matrix} are specified in \cite{BS}, (5.1) and (5.4), respectively. 

Employing the same  arguments as in \cite{St1}, p. 12-13, the Weil representation can in the case of $(kl,N)>1$ be extended to double cosets of the form $\Gamma m(l,k)\Gamma$ and $\Gamma m(k^{-1},l^{-1})\Gamma$. The action of these double cosets by means of \eqref{eq:hecke_op_copr}   yield the Hecke operators $T(m(l,k))$ and $T(m(k^{-1},l^{-1}))$ on the space $S_{\kappa}(\rho_L)$.

\begin{lemma}\label{lem:rel_hecke_ops_eigenvalues}
  Let $k,l\in \Z$ with $k|l$ and $\frac{l}{k}$ a square. The Hecke operators  $T(m(k^{-1},l^{-1})$ and  $T(m(l,k))$ are related to $T(m(\frac{l}{k},1))$ by
  \begin{align*}
    &  T(m(l,k)) = k^{\kappa-2}\frac{g_k(D(k))}{g(D(k))}T(m(\frac{l}{k},1)), \\
    & T(m(k^{-1},l^{-1})) = l^{2-\kappa}\frac{g(D(l))}{g_l(D(l))}T(m(\frac{l}{k},1)).
    \end{align*}
\end{lemma}
\begin{proof}
  Let $\Gamma m(l,k)\Gamma = \bigcup_i \Gamma \alpha_i$ be a decomposition of $\Gamma m(l,k)\Gamma$ into left cosets. Then $\bigcup_i \Gamma m(k,k)^{-1}\alpha_i$ is a decomposition of $\Gamma m(\frac{l}{k},1) \Gamma$ into left cosets.
  Consequently, using \eqref{eq:hecke_op_copr} we obtain
  \[
  T(m(l,k)) = k^{\kappa-2}\frac{g_k(D(k))}{g(D(k))}T(m(\frac{l}{k},1)).
  \]
  The proof of the identity for $T(m(k^{-1},l^{-1}))$ procceeds the same way.  
  \end{proof}

As a corollary we immediately get
\begin{corollary}\label{cor;rel_eigenvalues}
  Let $k,l\in \Z$ with $k|l$ and $\frac{l}{k}$ a square. Also, let $f\in S_{\kappa,L}$ be simultaneous eigenform of all Hecke operators $T(m(d^2,1))$ with eigenvalue $\lambda(m(d^2,1))$. Then $f$ is also an eigenform with respect to the Hecke operators $T(m(k^{-1},l^{-1}))$ and $T(m(l,k))$ with eigenvalues
  \begin{equation}\label{eq:rel_eigenvalues}
    \lambda(m(k^{-1},l^{-1})) = l^{2-\kappa}\frac{g(D(l))}{g_l(D(l))}\lambda(m(\frac{l}{k},1)) \text{ and } \lambda(m(l,k)) = k^{\kappa-2}\frac{g_k(D(k))}{g(D(k))}\lambda(m(\frac{l}{k},1))
  \end{equation}
  respectively and vice versa. 
  \end{corollary}

 Let $f\in S_\kappa(\rho_L)$ be a simultaneous eigenform of all Hecke operators $T(m(d^2,1)),\; d\in\N$, with eigenvalues $\lambda_f(m(d^2,1))$ (see Remark 6.1, ii) in \cite{St1} for some details when such a cusp form exists).
In \cite{St1} the analytic properties of the {\it standard zeta function}
\begin{equation}\label{eq:zeta_func_eigenvalues}
Z(s,f)= \sum_{d\in \N}\lambda_f(m(d^2,1))d^{-2s}
\end{equation}
were studied. By Theorem 5.6 in \cite{BS} this series possesses an Euler product
\begin{equation}\label{eq:euler_product}
    Z(s,f) = \prod_p Z_p(s,f) 
\end{equation}
with $Z_p(s,f) =\sum_{k\in \N_0}\lambda_f(m(p^{2k},1))p^{-2ks}.$ Let us consider the related local zeta function
\begin{equation}\label{eq:ext_local_zeta}
  \calZ_p(s,f)=\sum_{(k,l)\in \Lambda_+}\lambda_{f}(m(p^{-k},p^{-l}))p^{-s(k+l)}.
\end{equation}
In view of Corollary \ref{cor;rel_eigenvalues} we have
\begin{corollary}\label{cor:rel_local_zeta}
  The local zeta function $\calZ_p(s,f)$ is equal to
  \begin{equation}\label{eq:rel_local_zeta}
    \begin{split}
      &  \sum_{(k,l)\in \Lambda_+}p^{l(2-\kappa)}\frac{g(D(p^l))}{g_{p^l}(D(p^l))}\lambda(m(p^{l-k},1))p^{-s(k+l)} \\
      &= \sum_{(k,l)\in \Lambda_+}\frac{g(D(p^l))}{g_{p^l}(D(p^l))}\lambda(m(p^{l-k},1))p^{-sk}p^{(\kappa-2-s)l}.
    \end{split}
    \end{equation}
  \end{corollary}
Subsequently, we want to relate the zeta functions $Z(s,f)$ and $\calZ(s,f)$. To distinguish the cases of ``good'' and ``bad'' primes $p$, we will evaluate the quotients $\displaystyle \frac{g(D(p^l))}{g_{p^l}(D(p^l))}$ explicitly.
To this end, we assume that level $N$  of $L$ is  {\it square free}. We keep this assumption whenever we are dealing with the zeta function $\calZ(s,f)$.
By \cite{We1}, Lemma 5.8, we know that
\begin{equation}\label{eq:quadratic_char_gauss_sum_qout}
  n\mapsto \chi_{D}(n)=\frac{g_n(D)}{g(D)} 
\end{equation}
is a quadratic character of $(\Z/N\Z)^\times$. More specifically, by Lemma \ref{lem:gauss_sum_prop}, ii) 
\begin{equation}\label{eq:explicit_quadratic_char}
  \frac{g(D)}{g_n(D)}  =  \frac{g_n(D)}{g(D)} = \leg{n}{|D|}e\left(\frac{(n-1)\text{oddity}(D)}{8}\right).
  \end{equation}
A proof for the last equation can be found in \cite{We1}, Theorem 5.17. It is known that if  $|D|$ is odd,  $\text{ oddity}(D)\equiv 0\bmod{8}$, see e. g. \cite{We1}, Lemma 5.8 or \cite{CS}, Chap. 15, $\S$ 7. Thus, $\displaystyle \frac{g(D)}{g_n(D)}$ simplifies to 
\begin{equation}\label{eq:explicit_quadratic_char_D_odd}
 \chi_{D}(n)  = \leg{n}{|D|}
\end{equation}
in this case. Consequently, we have
\begin{equation}\label{eq:explicit_gauss_sum_quot}
  \frac{g(D(p^l))}{g_{p^l}(D(p^l))} = \chi_{D(p^l)}(p^l)
\end{equation}
for all $l\in \N_0$. Notice that if $p^l$ is a square, $g_{p^l}(D(p^l))= g(D(p^l)))$ and $\chi_{D(p^l)}(p^l)=1$. Also notice that for a prime $p$ the group $D(p^l)$ is equal to $D(p) = D_p^\perp$. For the subsequent calculations we will therefore use the symbol $\chi_{D(p)}$ instead of $\chi_{D(p^l)}$. 

For a prime $p$ dividing  $N$ 
the right-hand side of \eqref{eq:rel_local_zeta} becomes
\begin{equation}\label{eq:local_zeta_bad_prime}
  \begin{split}
 \sum_{(k,l)\in \Lambda_+}\chi_{D(p)}(p^l)\lambda_f(m(p^{l-k},1))p^{-l(\kappa-2)}p^{-s(k+l)}.
    \end{split}
\end{equation}

  For a ``good'' prime $p\nmid N$ we have
  \begin{equation}\label{eq:local_zeta_good_prime}
  \calZ_p(s,f) = \sum_{(k,l)\in \Lambda_+}\chi_{D}(p^l)\lambda_f(m(p^{l-k},1))p^{-l(\kappa-2)}p^{-s(k+l)}.
  \end{equation}
  Let $\chi$ be either of the characters $\chi_{D}$, $\chi_{D(p)}$. Then we can write for either of the above sums:
  \begin{align*}
    \sum_{(k,l)\in \Lambda_+}\chi(p^l)p^{-k(2s+\kappa-2)}\lambda_f(m(p^{l-k},1))p^{-(l-k)(s+\kappa -2)}.
  \end{align*}
  For any fixed $k\in \N_0$ the index $l$ runs through the set $\{2n+k\;|\; n\in \N_0\}$ to satisfy the conditions $l+k\in 2\Z$ and $l\ge k$. Therefore, the index $l-k$ runs through
  \[
  \{2n\; |\; n\in \N_0 \} = 2\N_0. 
  \]
  Thus, we may rewrite the latter series above as
  \begin{align*}
    &  \sum_k p^{-k(2s+\kappa-2)}\sum_n\chi(p^{2n+k})\lambda_f(m(p^{2n},1))p^{-2n(s+\kappa-2)} \\
    &= \sum_{k}\chi(p^k)p^{-k(2s+\kappa-2)}\sum_{n}\lambda_f(m(p^{2n},1))p^{-2n(s+\kappa -2)}. 
  \end{align*}
  Taking this into account, we may write \eqref{eq:local_zeta_bad_prime} in the form 
  \begin{equation}\label{eq:local_zeta_bad_prime_II}
    Z_p(s+\kappa-2,f)L_p(\chi_{D(p)},2s+\kappa-2). 
    \end{equation}
  Accordingly, for a ``good'' prime, \eqref{eq:local_zeta_good_prime} can be expressed as  $Z_p(s+\kappa -2, f)L_p(\chi_D,2s+\kappa-2)$. 
 
  Globally, we then have
  \begin{equation}\label{eq:global_general_zeta_func}
    \calZ(s,f) = \prod_{p\mid |D|}L_p(2s+\kappa-2,\chi_{D(p)})L(2s+\kappa-2,\chi_D)Z(s+\kappa-2,f),
  \end{equation}
  where
  \begin{enumerate}
    \item[i)]
      \[
      L_p(s,\chi_{D(p)})  =(1-\chi_{D(p)}(p)p^{-s})^{-1},
      \]
    \item[ii)]
      $L(s,\chi_D)$  is the Dirichlet $L$-series associated to $\chi_D$. 
  \end{enumerate}

  We can now state  the following theorem regarding the analytic properties of $\calZ(s,f)$.
  \begin{theorem}\label{thm:analytic_propo_zeta_func}
    Let $\kappa\in 2\Z,\; \kappa\ge 3,$ satisfy $2\kappa +\sig(L)\equiv 0\bmod{4}$. Moreover, assume that level $N$ of $L$ is square free  and let $f\in S_\kappa(\rho_L)$ a common eigenform of all Hecke operators $T(m(k^2,1))$, $k\in \N$. Then the zeta function $\calZ(s,f)$ has a meromorphic continuation to the whole $s$-plane. 
  \end{theorem}
  \begin{proof}
    This results from \eqref{eq:global_general_zeta_func}: Theorem 6.6 of \cite{St1} provides the desired property for $Z(s,f)$. In fact, the slightly altered action of the Weil representation on scalar matrices only affects the pullback formula in Theorem 5.3 but not the proof of Theorem 6.6 as the factor $\displaystyle \frac{g_d(L)}{g(L)}$ cancels out. For $L(s,\chi_D)$  this is well known and clear for the remaining factor  anyway.  
  \end{proof}

\section{Standard $L$-function of a common Hecke eigenform}\label{subsec:standard_L_func}

This chapter is concerned with several issues regarding a standard $L$-function of a vector valued Hecke eigenform. First, we will motivate and define a standard $L$-function $L(s,F)$ attached to a vector valued automorphic form $F$. Its definition is based on the factorization of the Hecke series. To this end, we follow the classical contributions in the literature in this regard (see e. g. the approach of Bouganis and Marzec (\cite{BM}), B\"ocherer and Schulze-Pillot (\cite{BoSP}) and Shimura (\cite{Sh})) adapted to our situation. Via the correspondence in \cite{St3}, Theorem 5.9, it is then possible to associate the same $L$-function to the corresponding Hecke eigenform $f_F$. Along the way we establish most of the results to prove a relation between the standard zeta function $\calZ(s,f)$ and $L(s,f)$. Afterwards, based on this relation, we  prove that the introduced standard $L$-function can be continued meromorphically to the whole $s$-plane.  

We assume in the whole Section that the order of $D$ is odd (and therefore the level $N$ of $L$ is odd) and that $p$-group $D_p$ is anisotropic for each $p$ dividing $|D|$. We adopt the notation of \cite{St3}, Section 3. 

\subsection{Standard $L$-function of a vector valued automorphic form}
This section provides the necessary theory to define a standard $L$-function attached to a vector valued automorphic form  $F$ as defined in \cite{St3}. This is essentially the well known theory of spherical functions as it appears in many places (see e. g. \cite{Ca}, \cite{McD} or \cite{Sa}).
Here we largely follow \cite{Ar}, Chapter 5, and translate several statements therein to our setting. For primes $p$ which are coprime to $|D|$ the proofs carry over almost verbatim. In the case of primes $p$ dividing $|D|$ the Hecke algebra $\calH^+(\calQ_p//\calK_p,\omega_p)$ is more complicated due to the Weil representation, causing serious difficulties. The established results then provide  the means to define  a standard $L$-function $L(s,F)$ along the lines of  \cite{BM}, Chapter 7.2. 

The action of the local  Weil representation of the group $\calK_0(p)$ on $S_{L_p}^{N(\Z_p)}$ will be needed frequently in the proof of the next few assertions.

\begin{lemma}\label{lem:weil_repr_K_0_p}
  Let $\kzxz{a}{b}{c}{d}\in \calK_0(p)$. Then
  \begin{equation}\label{eq:decomp_K_0_p}
  \kzxz{a}{b}{c}{d} = n\_(ca^{-1})m(a,a^{-1}(ad-bc))n(ba^{-1})
  \end{equation} 
  and
  \begin{equation}\label{eq:weil_repr_K_0_p}
    \omega_p\kzxz{a}{b}{c}{d}\varphi_p^{(0)} = \chi_{D_p}(a)\varphi_p^{(0)}.
    \end{equation}
\end{lemma}
\begin{proof}
  First note that if $c\in p\Z_p$, then $a$ must be an element of $\Z_p^\times$. The decomposition \eqref{eq:decomp_K_0_p} can be checked by a direct calculation.

  To compute $\omega_p\kzxz{a}{b}{c}{d}\varphi_p^{(0)}$, we use the defining  formulas of $\omega_p$ (cf. Section \ref{subsec:local_weil} and \cite{St3}, Theorem 4.8 for matrices in $\calD_p$):  
  \begin{align*}
    \omega_p\kzxz{a}{b}{c}{d}\varphi_p^{(0)} &= \omega_p(n\_(ca^{-1}))\omega_p(m(a,a^{-1}(ad-bc)))\omega_p(n(ba^{-1}))\varphi_p^{(0)} \\
    & = \omega_p(n\_(ca^{-1}))\omega_p(m(a,a^{-1}(ad-bc)))\varphi_p^{(0)}\\
    & = \omega_p(n\_(ca^{-1}))\leg{a}{|D_p|}\varphi_p^{(0)}\\
    & = \leg{a}{|D_p|}\varphi_p^{(0)}.
  \end{align*}
  For the second equation we employed that $\psi_p(ca^{-1}q(0)) = 1$,  the last equation is due to \cite{St3}, Lemma 3.1. 
  \end{proof}

\begin{lemma}\label{lem:algebra_hom}
  Let $p$ be a prime.
  \begin{enumerate}
  \item[i)]
If $p$ is  coprime to $|D|$,  any $\C$-algebra homomorphism $\xi: \calH(\calQ_p//\calK_p, \omega_p)\rightarrow \C$ is of the form
  \begin{equation}\label{eq:hecke_algebra_hom}
   T\mapsto \xi(T) = \widehat{\chi}_S(T),
  \end{equation}
  Here $\chi$ is some unramified character of $\calM_p$ determined up the Weyl action and
  \begin{equation}\label{eq:character_series}
    \widehat{\chi}_S(T) = \sum_{(k,l)\in \Lambda} S(\langle T\varphi_p^{(0)}, \varphi_p^{(0)}\rangle)(m(p^k,p^l))\chi(m(p^k,p^l)),
  \end{equation}
  $S$ being the classical Satake map.
\item[ii)]
  If $p$ divides $|D|$, any $\C$-algebra homomorphism  $\xi:\calH^+(\calQ_p//\calK_p, \omega_p)\rightarrow \C$ is of the form
  \begin{equation}\label{eq:hecke_algebra_hom_ncopr_nsquare}
    \begin{split}
      T\mapsto \xi(T) = \widehat{\chi}_\calS(T)&=\sum_{(k,l)\in \Lambda} \langle (I_{\chi_{D_p}}\circ\calS(T))(m(p^k,p^l))\varphi_p^{(0)}, \varphi_p^{(0)}\rangle)\chi(m(p^k,p^l)) \\
      & = \sum_{(k,l)\in \Lambda} \langle\calS(T)(m(p^k,p^l))\varphi_p^{(0)}, \varphi_p^{(0)}\rangle)\chi(m(p^k,p^l)),
    \end{split}
    \end{equation}
  \end{enumerate}
  where $\calS$ is the Satake map given in \cite{St3}, (4.15),  $I_{\chi_{D_p}}$ is the isomorphism in \cite{St3}, (4.13), and $\chi$ is again an unramified character of $\calM_p$.
  \end{lemma}
\begin{proof}
First, recall that $\widehat{\chi}_S$ and $\widehat{\chi}_{\calS}$ are well defined since $S(\langle T,\varphi_p^{(0)}\rangle)$ and $\langle \calS(T),\varphi_p^{(0)}\rangle$ have finite support on $\calM_p$. 
  
i)   As $\omega_p$ is trivial in this case, $f \mapsto \langle f, \varphi_p^{(0)}\rangle$ is an isomorphism of the Hecke algebras $\calH(\calQ_p//\calK_p, \omega_p)$ and $\calH(\calQ_p//\calK_p)$. From \cite{Ca}, Corollary 4.2, we know that any algebra homomorphism of $\calH(\calQ_p//\calK_p)$ is given by
  \[
  g\mapsto \int_{t\in \calM_p}S(g)(t)\chi(t)dt,
  \]
  where $\chi$ is an unramified character of $\calM_p$. Since $S(g)$ is  bi-invariant under $\calD_p$ and $\chi$ unramified, we obtain the above stated term.
  (note that we adopted the normalisation of the Haar measures in \cite{Ca}, see Section \ref{sec:notation}).
  
  ii)  The proof in i) essentially relies on the fact that any $\C$-algebra homomorphism $\xi$ of the group algebra $\C[\calM_p/\calD_p]$ can be written in terms of a uniquely determined unramified character $\chi$ of $\calM_p$ by
  \[
  \xi(T)= \sum_{m\in \calM_p/\calD_p}T(m)\chi(m).
  \]
  By \cite{St3}, Theorem 4.10, Theorem 4.8, ii), and the remark directly after its proof, we know  that
  \[
  T\mapsto \langle I_{\chi_{D_p}}\circ \calS(T)\varphi_p^{(0)}, \varphi_p^{(0)}\rangle
  \]
uniquely determined   maps $\calH^+(\calQ_p//\calK_p, \omega_p)$ isomorphically to a subalgebra of $\C[\calM_p/\calD_p]^W$ inducing the claimed form of $\xi$. 
  \end{proof}
Note that each unramified character $\chi$ of $\calM_p$ is of the form $\chi(m(t_1,t_2))= \chi_1(t_1)\chi_2(t_2)$, where $\chi_i$ is an umramified character of $\Q_p^\times$, that is, $\chi_i$ is trivial on $\Z_p$.

Attached to an  unramified character $\chi$ of $\calM_p$ we now introduce  a operator valued  map on $\calQ_p$. The scalar valued version is part of the classical zonal spherical function associated to $\chi$. (cf. e. g. \cite{Ca}, p. 150): 
Let
\begin{equation}\label{def:modulus_char}
  \begin{split}
    &  \phi_\chi: \calQ_p\rightarrow \Hom(S_{L_p}^{N(\Z_p)}, S_{L_p}),\quad  \\
    & g=mnk\mapsto \left(\phi_\chi(mnk)\varphi_p^{(0)}= (\chi_{D_p}\chi\delta^{\frac{1}{2}})(m)\omega_p(k)\varphi_p^{(0)}\right),  n\in N(\Q_p), m\in \calM_p, k\in \calK_p,
  \end{split}
  \end{equation}
where  $\delta(m(t_1,t_2))=\left|\frac{t_1}{t_2}\right|_p$ is the modulus character and $\chi_{D_p}$ the induced Hecke character on $\calK_p$ by \eqref{eq:quadratic_char_gauss_sum_qout} and extended trivially to $\calM_p$.
The next lemma is in principle well known for primes $p$ coprime to $|D|$. 

\begin{lemma}\label{lem:zonal_spherical_func_fourier}
  Let $p$ be a prime, $\mu(\calK_0(p))$ the measure of $\calK_0(p)$, 
  $T\in  \calH(\calQ_p//\calK_p,\omega_p)$ and
\[
  \kappa_p = \left(\frac{|D_p|p}{\gamma_p(D_p)^2}+1\right)\mu(\calK_0(p)).
\]
  Then the following identities hold:
\begin{equation}\label{eq:fourier_tramsform_coprime}
  \int_{\calQ_p}\langle T(g),\varphi_p^{(0)}\rangle \langle \phi_\chi(g)\varphi_p^{(0)}, \varphi_p^{(0)}\rangle dg =
  \widehat{\chi}_S(T)
\end{equation}
if $p$ is coprime to $|D|$.
\begin{equation}\label{eq:fourier_transform_ncopr}
 \int_{\calQ_p}\langle T(g)\varphi_p^{(0)},\varphi_p^{(0)}\rangle\langle \phi_\chi(g)\varphi_p^{(0)}, \varphi_p^{(0)}\rangle dg = \kappa_p\widehat{\chi}_\calS(T)
\end{equation}
if $p$ is a divisor of  $|D|$.
\end{lemma}

\begin{proof}
If $(p,|D|)=1$, the algebra $\calH(\calQ_p//\calK_p,\omega_p)$ is isomorphic to the classical algebra $\calH(\calQ_p//\calK_p)$. In particular,  any $T\in \calH(\calQ_p//\calK_p,\omega_p)$ is bi-invariant with respect to $\calK_p$.
Therefore, the  result can be proved as in \cite{McD}, p. 46, or \cite{Ca}, p. 150. Following the proof of either of the cited sources, we end up with
  \begin{equation}\label{eq:comp_fourier_tramsform}
    \int_{m\in \calM_p}\chi(m)\delta(m)^{1/2}\int_{N(\Q_p)}\langle T(mn),\varphi_p^{(0)}\rangle dn.
  \end{equation}
  Since $S(\langle T,\varphi_p^{(0)}\rangle)$ is bi-invariant under $\calD_p$ and $\chi$ is unramified, this is equal to
  \[
  \sum_{(k,l)\in \Lambda}\chi(m(p^k,p^l))S(\langle T,\varphi_p^{(0)}\rangle)(m(p^k,p^l)).
  \]
  For $p\mid |D|$, the computations are more  involved since the integrand is not bi-invariant under $\calK_p$. However, it is still possible to remedy the absence of the  bi-invariance. To this end, we use two facts. First, due to Lemma \ref{lem:weil_repr_K_0_p} we know that the Weil representation $\omega_p(k')$ acts  on $S_{L_p}^{N(\Z_p)}$ by multiplication with the quadratic character $\chi_{D_p}$ for any $k'\in \calK_0(p)$. Secondly, we use Lemma 13.1 of \cite{KL} (whose proof is still valid for the groups $\calK_p$ and $\calK_0(p)$), which provides with $\{T^{j}w^{-1}\; |\; j\in \Z/p\Z\}\cup \{1_2\}$ a set of coset representatives of $\calK_p/\calK_0(p)$ allowing us to calculate $\omega_p$ on $\calK_p$ explicitly. As a consequence, we write the integral $\int_{\calK_p}$ in the form  $\displaystyle \int_{\calK_p/\calK_0(p)}\int_{\calK_0(p)} = \sum_{\calK_p/\calK_0(p)}\int_{\calK_0(p)}$. 
Thus, taking this and the Iwasawa decomposition into account, we have
  \begin{equation}\label{eq:interal_expressions}
    \begin{split}
      &   \int_{\calQ_p}\langle T(g)\varphi_p^{(0)},\varphi_p^{(0)}\rangle \langle \phi_\chi(g)\varphi_p^{(0)}, \varphi_p^{(0)}\rangle dg \\
      &= \sum_{(k,l)\in \Lambda}\int_{N(\Q_p)}\int_{\calK_p}\langle T(m(p^k,p^l)n)\omega_p(k)\varphi_p^{(0)},\varphi_p^{(0)}\rangle\langle (\chi_{D_p}\chi\delta^{1/2})(m(p^k,p^l))\omega_p(k)\varphi_p^{(0)},\varphi_p^{(0)} \rangle dk\;dn \\
      & =\sum_{(k,l)\in \Lambda}\int_{N(\Q_p)}\int_{\calK_0(p)}\times \\
      &\langle T(m(p^k,p^l)n)\omega_p(k')\varphi_p^{(0)},\varphi_p^{(0)}\rangle\langle (\chi\delta^{1/2})(m(p^k,p^l))\omega_p(k')\varphi_p^{(0)},\varphi_p^{(0)} \rangle dk'\;dn \\
      &+ \sum_{(k,l)\in \Lambda}\int_{N(\Q_p)}\sum_{j\in \Z/p\Z}\int_{\calK_0(p)}\times \\
      &\langle T(m(p^k,p^l)n)\omega_p(T^{-j}w)\omega_p(k')\varphi_p^{(0)},\varphi_p^{(0)}\rangle\langle (\chi\delta^{1/2})(m(p^k,p^l))\omega_p(T^{-j}w)\omega_p(k')\varphi_p^{(0)},\varphi_p^{(0)} \rangle dk'\;dn.\\
      \end{split}
  \end{equation}
  Here, the contribution of the character $\chi_{D_p}$ in the definition of $\phi_\chi$ disappears as it is trivial on matrices $m(p^k,p^l)$. 
In light of the decomposition \eqref{eq:decomp_K_0_p} of any matrix $k'=\kzxz{a}{b}{c}{d}\in \calK_0(p)$,  we may  replace
  \[
  \int_{\calK_0(p)} \text{ with } \int_{(\Z_p^\times)^2}\int_{\Z_p^\times}\int_{\Z_p}\int_{\Z_p}
  \]
  by Fubini's theorem. 
  Thus, by Lemma \ref{lem:weil_repr_K_0_p}
  \begin{equation}\label{eq:integral_K_0_p}
    \begin{split}
    &\int_{\calK_0(p)}\langle T(m(p^k,p^l)n)\omega_p(T^{-j}w)\omega_p(k')\varphi_p^{(0)},\varphi_p^{(0)}\rangle\langle (\chi\delta^{1/2}(m(p^k,p^l))\omega_p(T^{-j}w)\omega_p(k')\varphi_p^{(0)},\varphi_p^{(0)} \rangle dk \\
    &= \int_{(\Z_p^\times)^2}\int_{\Z_p^\times}\int_{\Z_p}\int_{\Z_p}\times\\
&   \langle T(m(p^k,p^l)n)\omega_p(T^{-j}w)\chi_{D_p}(a)\varphi_p^{(0)},\varphi_p^{(0)}\rangle\langle \phi_\chi(m(p^k,p^l))\omega_p(T^{-j}w)\chi_{D_p}(a)\varphi_p^{(0)},\varphi_p^{(0)}\rangle dc\;db\:da\;dr.
      \end{split}
    \end{equation}
  As quadratic character, the $\chi_{D_p}(a)$ in both scalar products $\langle\cdot,\cdot\rangle$ cancel out and we obtain for the right-hand side of \eqref{eq:integral_K_0_p}
  \[
  \mu(\calK_0(p))\langle T(m(p^k,p^l)n)\omega_p(T^{-j}w)\varphi_p^{(0)},\varphi_p^{(0)}\rangle\langle (\chi\delta^{1/2})(m(p^k,p^l))\omega_p(T^{-j}w)\varphi_p^{(0)},\varphi_p^{(0)}\rangle
  \]
  The formulas (2.7) in \cite{St3} for $\omega_p$ allow us to evaluate the sum over $j$ explicitly:
\begin{equation}\label{eq:weil_T_w}
  \omega_p(T^{j}w^{-1})\varphi_p^{(0)} = \frac{|D_p|^{1/2}}{\gamma_p(D_p)}\sum_{\gamma_p\in D_p}e(jq(\gamma_p))\varphi_p^{(\gamma_p)}. 
  \end{equation}
  Thus, the applying the definition of the second scalar product yields
  \begin{align*}
  & \langle T(m(p^k,p^l)n)\omega_p(T^{j}w^{-1})\varphi_p^{(0)},\varphi_p^{(0)}\rangle\langle (\chi\delta^{1/2})(m(p^k,p^l))\omega_p(T^{j}w^{-1})\varphi_p^{(0)},\varphi_p^{(0)}\rangle \\
  &= \frac{|D_p|}{\gamma_p(D_p)^2}\sum_{\gamma_p,\mu_p\in D_p}\psi_p(j(q(\gamma_p)+q(\mu_p)))\langle T(m(p^k,p^l)n)\varphi_p^{(\gamma_p)},\varphi_p^{(0)}\rangle\langle (\chi\delta^{1/2})(m(p^k,p^l))\varphi_p^{(\mu_p)},\varphi_p^{(0)}\rangle \\
  &= \frac{|D_p|}{\gamma_p(D_p)^2}\sum_{\gamma_p\in D_p}\psi_p(j(q(\gamma_p))\langle T(m(p^k,p^l)n)\varphi_p^{(\gamma_p)},\varphi_p^{(0)}\rangle (\delta^{1/2}\chi)(m(p^k,p^l)).
  \end{align*}
By means of the standard Gauss sum identity applied to the sum over $j\in \Z/p\Z$,  we finally find 
\begin{equation}\label{eq:int_K_p}
  \begin{split}
  & \int_{\calK_p}\langle T(m(p^k,p^l)n)\omega_p(k)\varphi_p^{(0)},\varphi_p^{(0)}\rangle\langle \phi_\chi(m(p^k,p^l))\omega_p(k)\varphi_p^{(0)},\varphi_p^{(0)} \rangle dk\\
  & = \left(\frac{|D_p|p}{\gamma_p(D_p)^2}+1\right)\mu(\calK_0(p))\langle T(m(p^k,p^l)n)\varphi_p^{(0)},\varphi_p^{(0)}\rangle (\delta^{1/2}\chi)(m(p^k,p^l))
  \end{split}
\end{equation}
and therefore
  
  \begin{align*}
    &  \int_{\calQ_p}\langle T(g)\varphi_p^{(0)},\varphi_p^{(0)}\rangle \langle \phi_\chi(g)\varphi_p^{(0)},\varphi_p^{(0)}\rangle dg \\
    &= \kappa_p\sum_{(k,l)\in \Lambda}\delta(m(p^k,p^l))^{1/2}\chi(m(p^k,p^l)))\int_{N(\Q_p)}\langle T(m(p^k,p^l)n)\varphi_p^{(0)},\varphi_p^{(0)}\rangle dn  \\
    & = \kappa_p\sum_{(k,l)\in \Lambda}\langle \calS(T)(m(p^k,p^l))\varphi_p^{(0)},\varphi_p^{(0)}\rangle \chi(m(p^k,p^l))
  \end{align*}
  as $T_{|_{S_{L_p}^{N(\Z_p)}}}$ is right-invariant under $N(\Z_p)$.
\end{proof}

To relate $L(s,F)$ and $\calZ(s,f_F)$, we calculate the integral
\[
\int_{\calQ_p}\langle\nu_s(g)\varphi_p^{(0)},\varphi_p^{(0)}\rangle\langle \phi_\chi(g)\varphi_p^{(0)},\varphi_p^{(0)}\rangle dg
\]
in two different ways. The first one is an analogue of  Lemma 5.2 in \cite{Ar}.
Here for $s\in \C$ we define $\nu_s:\calQ_p\rightarrow \End(S_{L_p})$ by
\begin{equation}\label{eq:nu_s}
  \begin{split}
  & \nu_s(k_1gk_2) = \omega_p(k_1)\nu_s(g)\omega_p(k_2) \text{ for all } k_1,k_2\in \calK_p \text{ and all } g\in \calQ_p, \\
  & \nu_s(m(p^k,p^l))\varphi_p^{(\gamma_p)} = 
  \begin{cases}
    p^{-(k+l)s}\varphi_p^{(p^{(l-k)/2}\gamma_p)}, &  (k,l)\in \Lambda_+,\\
    0, &  k \text{ or }  l < 0.
    \end{cases}
  \end{split}
\end{equation}
The proof of the next result makes use of the following observation:
\begin{equation}\label{eq:T_k_l_n_s}
  \begin{split}
    \nu_s(k_1m(p^k,p^l)k_2) &= p^{-s(k+l)}T_{k,l}(k_1m(p^k,p^l)k_2)\\
    \end{split}
\end{equation}
for all $k_1,k_2\in \calK_p$ and all $(k,l)\in \Lambda_+$.

\begin{lemma}\label{lem:integral_series}
  Let $p$ be a prime,  $\chi$ be an unramified character of $\calM_p$,
\[
  T_{k,l}\in
  \begin{cases}\calH(\calQ_p//\calK_p,\omega_p), & p\nmid |D|,\\
    \calH^+(\calQ_p//\calK_p,\omega_p), & p\mid |D|,
  \end{cases}
\]
  and
\[
B_S(\chi, X)=
\sum_{(k,l)\in \Lambda_+}\widehat{\chi_S}(T_{k,l})X^{k+l} \text{ and }  B_\calS(\chi, X)=
\sum_{(k,l)\in \Lambda_+}\widehat{\chi_\calS}(T_{k,l})X^{k+l}
  \]
  where  $\widehat{\chi_S}$, $\widehat{\chi_\calS}$ are defined in Lemma \ref{lem:algebra_hom}. 
Then
\begin{enumerate}
\item[i)]
    \begin{equation}\label{eq:dirichlet_series_coprime}
    \int_{\calQ_p}\langle \nu_s(g)\varphi_p^{(0)},\varphi_p^{(0)}\rangle\langle\phi_\chi(g)\varphi_p^{(0)},\varphi_p^{(0)}\rangle dg = B(\chi,p^{-s}) 
    \end{equation}
    if $p\nmid |D|$.
  \item[ii)]
    \begin{equation}\label{eq:dirichlet_series_ncoprime}
    \int_{\calQ_p}\langle\nu_s(g)\varphi_p^{(0)},\varphi_p^{(0)}\rangle\langle \phi_\chi(g)\varphi_p^{(0)},\varphi_p^{(0)}\rangle dg = \kappa_pB(\chi,p^{-s}) 
    \end{equation}
    if $p\mid |D|$, where $\kappa_p$ is specified in Lemma \ref{lem:zonal_spherical_func_fourier}. 
\end{enumerate}
\end{lemma}
\begin{proof}
  For $p$ coprime to $|D|$ the proof is essentially the same as the one of Lemma 5.2 in \cite{Ar}. One has just to replace the term $\nu_s$ with our corresponding $\nu_s$ and $\varphi_\alpha$ with $T_{k,l}$.

  Bearing \eqref{eq:T_k_l_n_s} in mind and the fact that the operator $T_{k,l}$ is supported on $\calK_pm(p^k,p^l)\calK_p$, we find that pointwise
  \begin{equation}\label{eq:T_k_l_nu_s}
    \nu_s = \sum_{(k,l)\in \Lambda_+}p^{-s(k+l)}T_{k,l}.
    \end{equation}
Plugging this into \eqref{eq:dirichlet_series_ncoprime} yields  
\begin{equation}
  \sum_{(k,l)\in \Lambda_+}p^{-s(k+l)}\int_{\calQ_p}\langle T_{k,l}(g)\varphi_p^{(0)},\varphi_p^{(0)}\rangle\langle\phi_\chi(g)\varphi_p^{(0)},\varphi_p^{(0)}\rangle dg.
\end{equation}
In light of Lemma \ref{lem:zonal_spherical_func_fourier} we obtain the result.
\end{proof}

Observe that in view of Lemma \ref{lem:int_repr_series}, it is guaranteed that the  Dirichlet series $B(\chi, p^{-s})$ converges in both considered instances  normally for all $s\in \C$ with $\re(s)$ sufficiently large and represents in the region of convergence a holomorphic function. The equations \eqref{eq:dirichlet_series_coprime} and \eqref{eq:dirichlet_series_ncoprime} are valid for these $s\in \C$ and each integral on the left-hand side of these equations is consequently a holomorphic function in $s$ on the before mentioned region (in fact, they are valid for all $s$ where the right-hand side of \eqref{eq:rel_int_local_L_func_cpr} is defined).  
The next assertion calculates the Satake map on the generators $T_{k,l}$ of the Hecke algebra $\calH(\calQ_p//\calK_p,\omega_p)$ under the assumption that $p$ does not divide $|D|$. This formula will be part of the proof of a subsequent result. 
\begin{lemma}\label{lem:explicit_satake_coprime}
  Let $p$ coprime to $|D|$ and $k<l$ and $\tau_{k,l}$ be defined as in \cite{St3}, (4.17). Then 
  \begin{equation}\label{eq:explicit_satake_coprime}
    S(T_{k,l}) = p^{(l-k)/2}\left(\tau_{k,l} + \tau_{l,k}+(1-p^{-1})\sum_{\nu=k+1}^{l-1}\tau_{\nu,k+l-\nu}\right).
  \end{equation}
  Moreover, $S(T_{k,k}) = \tau_{k,k}$. 
  \end{lemma}
\begin{proof}
  The assertion can be extracted directly from the proof of \cite{St3}, Theorem 4.10 by taking the Weil representation to be trivial. The remaing parts of the proof, like the calculation of the support and the sets $\calU(\nu)$ and  $\calU^0(l)$ remain unchanged.

  For $k<\nu<l$ we have $|\calU(\nu)| = (p-1)p^{\nu-k-1}$ and multiplication by the modulus character 
  \[
  \delta(m(p^\nu,p^{l+k-\nu}))^{1/2} = p^{(l+k-2\nu)/2}
  \]
  gives $p^{(l-k)/2}(1-p^{-1})$ for the constants of the middle terms of $S(T_{k,l})$. The constants of the two endpoint contributions both equal $p^{(l-k)/2}$.

  Note that by identifying the characteristic function $\mathbbm{1}_{\calD_pm(p^k,p^l)\calD_p}\cdot \id_{S_{L_p}}$ with the characteristic function $\mathbbm{1}_{\calD_pm(p^k,p^l)\calD_p} = \mathbbm{1}_{m(p^k,p^l)\calD_p}$ we indeed obtain the formulas for the classical Satake map.
\end{proof}

The following elementary generating-series identity will be used in the next lemma. 
\begin{remark}\label{rem:generating_series}
  Let $x,y$ and $z$ be indeterminates. Then
  \begin{equation}\label{eq:generating_series}
    \sum_{r\ge 1}\left(\sum_{j=1}^{2r-1}x^jy^{2r-j}\right)z^r = \frac{xyz(1+xyz)}{(1-x^2z)(1-y^2z)}.
  \end{equation}
   This can be deduced from the formal identity
   \begin{equation}\label{eq:double_geometric_sum}
   \sum_{r=0}^\infty \left(\sum_{j=0}^r x^jy^{r-j}\right)z^r = \frac{1}{(1-xz)(1-yz)},
   \end{equation}
   which we obtain by expanding both factors on the right-hand side as a geometric series and calculate subsequently their Cauchy product. 
We can reduce the computation of \newline$\sum_{r\ge 0}\left(\sum_{j=0}^{2r} x^jy^{2r-j}\right)z^r$ to \eqref{eq:double_geometric_sum} by splitting $\sum_{j=0}^{2r} x^jy^{2r-j}$ into odd and even $j$: 
   \begin{align*}
     \sum_{j=0}^{2r} x^jy^{2r-j} = \sum_{j=0}^rx^{2j}y^{2r-2j} + \sum_{j=0}^{r-1}x^{2j+1}y^{2(r-j)-1}
     \end{align*}
   such that
   \begin{align*}
     \sum_{r\ge 0}\left(\sum_{j=0}^{2r}x^jy^{2r-j}\right)z^r &= \sum_{r\ge 0}\left(\sum_{j=0}^rx^{2j}y^{2r-2j}\right)z^r + \sum_{r\ge 1}\left(\sum_{j=0}^{r-1}x^{2a+1}y^{2(r-a)-1}\right)z^r. 
   \end{align*}
   The first double sum gives $\frac{1}{(1-x^2z)(1-y^2z)}$, the second, after the transformation $r=s+1$ and factoring out $xyz$, gives $\frac{xyz}{(1-x^2z)(1-y^2z)}$. Finally
   \begin{align*}
     \sum_{r\ge 1}\left(\sum_{j=1}^{2r-1}x^{j}y^{2r-j}\right)z^r &= \sum_{r\ge 1}\left(\sum_{j=0}^{2r}x^{j}y^{2r-j}\right)z^r - \sum_{r\ge 1}x^{2r}z^r-\sum_{r\ge 1}y^{2r}z^r \\
     &= \frac{1 + xyz}{(1-x^2z)(1-y^2z)} - 1 - \frac{x^2z}{1-x^2z}-\frac{y^2z}{1-y^2z} \\
     & = \frac{xyz(1+xyz)}{(1-x^2z)(1-y^2z)},
   \end{align*}
   which leads to the claimed result.
\end{remark}

The next Lemma is a variant of a Theorem which is due to Murase and Sugano (see \cite{Ar}, Theorem 5.3). It connects the series $B(\chi,p^{-s})$ with a rational expression 
in terms of a prime power and an unramified character $\chi=(\chi_1,\chi_2)$ of $\calM_p$.

\begin{lemma}\label{lem:int_repr_series}
  Let $p$ be a prime, $\chi = (\chi_1,\chi_2)$ be an unramified character of $\calM_p$. 
  Then
  \[
  \int_{\calQ_p}\langle\nu_{s}(g)\varphi_p^{(0)},\varphi_p^{(0)}\rangle\langle \phi_\chi(g)\varphi_p^{(0)},\varphi_p^{(0)}\rangle dg 
  \]
  is equal to
   \begin{equation}\label{eq:rel_int_local_L_func_cpr}
    \begin{split}
      & \kappa_p\frac{1-\chi_1(p)^2\chi_2(p)^2p^{-4s+2}}{(1-\chi_1(p)\chi_2(p)p^{-2s})(1-\chi_1(p)^2p^{-2s+1})(1-\chi_2(p)^2p^{-2s+1})} \text{ if } p||D| \text{ and } \ord_p(D_p)\in 2\Z+1,\\
      &\\
      &\kappa_p\frac{(1+p\chi_1(p)\chi_2(p)p^{-2s})(1+(\gamma_p(D_p)(p-1)-1)p\chi_1(p)\chi_2(p)p^{-2s})}{(1-\chi_1(p)\chi_2(p)p^{-2s})(1-\chi_1(p)^2p^{-2s+1})(1-\chi_2(p)^2p^{-2s+1})} \text{ if } p||D| \text{ and } \ord_p(D_p)\in 2\Z,\\
      &\\
      & \frac{1+\chi_1(p)\chi_2(p)p^{-2s+1}}{(1-\chi_1(p)^2p^{-2s+1})(1-\chi_2(p)p^{-2s+1})} \text{ if } p\nmid |D|.
    \end{split}
    \end{equation}
  \end{lemma}
\begin{proof}
  Combining the steps  in the proof of Lemma \ref{lem:integral_series} with Lemma \ref{lem:zonal_spherical_func_fourier} we obtain
  \begin{align*}
    & \int_{\calQ_p}\langle\nu_{s}(g)\varphi_p^{(0)},\varphi_p^{(0)}\rangle\langle \phi_\chi(g)\varphi_p^{(0)},\varphi_p^{(0)}\rangle dg \\
    &= \sum_{(k,l)\in \Lambda_+}p^{-s(k+l)}\int_{\calQ_p}\langle T_{k,l}(g)\varphi_p^{(0)},\varphi_p^{(0)}\rangle\langle\phi_\chi(g)\varphi_p^{(0)},\varphi_p^{(0)}\rangle dg \\
    & = \begin{cases}
        \kappa_p\sum_{(k,l)\in \Lambda_+}p^{-s(k+l)}\widehat{\chi}_{\calS}(T_{k,l}), & p||D|,\\
        \sum_{(k,l)\in \Lambda_+}p^{-s(k+l)}\widehat{\chi}_{S}(T_{k,l}), & p\nmid |D|,
        \end{cases}
      \end{align*}
   where $\calS$ means the vector valued Satake transform and $S$ the classical one.  

   For the next steps we make use of the explicit formulas for the Satake transform to calculate the $\widehat{\chi}_{\mathscr{S}}(T_{k,l})$, where $\mathscr{S}$ represents either the transform  $\calS$ or $S$. By means of Lemma  \ref{lem:algebra_hom} we find
   \[
  \widehat{\chi}_{\mathscr{S}}(T_{k,l})= \sum_{(r,s)\in \Lambda} \langle\mathscr{S}(T_{k,l})(m(p^r,p^s))\varphi_p^{(0)}, \varphi_p^{(0)}\rangle)\chi(m(p^r,p^s))
   \]
   By Lemma \ref{lem:explicit_satake_coprime} and Theorem 4.10 in \cite{St3}, in particular formula (4.26), we may replace the right-hand side of the previous identity by
   \begin{align*}
     &  p^{(l-k)/2}\left[\chi(m(p^k,p^l))+\chi(m(p^l,p^k))+\sum_{\nu=k+1}^{l-1}c_{\nu, k+l-\nu}\chi(m(p^\nu,p^{k+l-\nu}))\right]\\
     & = p^{(l-k)/2}\left[\chi(m(p^k,p^l))+\chi(m(p^l,p^k))+c_p\sum_{\nu=k+1}^{l-1}\chi(m(p^\nu,p^{k+l-\nu}))\right],
   \end{align*}
   where we have used that $c_{\nu,k+l-\nu}$ is equal to some constant $c_p$ for all middle terms of the Satake transform. 
   If $k= l$, we simply have
   \[
   \chi(m(p^k,p^k)).
   \]
   Putting $l=k+2r$ and writing $\chi(m(p^u,p^v))= \chi_1(p)^u\chi_2(p)^v$ yields
   \begin{align*}
      \sum_{(k,l)\in \Lambda_+}\widehat{\chi}_{\mathscr{S}}(T_{k,l})p^{-s(k+l)}& = \sum_{k=0}^\infty\sum_{r=0}^\infty\widehat{\chi}_{\mathscr{S}}(T_{k,k+2r})p^{-2s(k+r)} \\
      &= \sum_{k=0}^\infty(\chi_1(p)\chi_2(p))^kp^{-2ks} + \sum_{k=0}^\infty\sum_{r=1}^\infty p^r\left[\chi_1(p)^k\chi_2(p)^{k+2r}+\chi_1(p)^{k+2r}\chi_2(p)^k\right.\\
        &\left. + c_p\sum_{\nu=k+1}^{k+2r-1}\chi_1(p)^\nu\chi_2(p)^{2k+2r-\nu}\right]p^{-2s(k+r)}\\
      &= \sum_{k=0}^\infty(\chi_1(p)\chi_2(p))^kp^{-2ks} + \sum_{k=0}^\infty(\chi_1(p)\chi_2(p))^kp^{-2ks}\sum_{r=1}^\infty\left[\chi_2(p)^{2r}+\chi_1(p)^{2r}\right.\\
        &\left.+c_p\sum_{j=1}^{2r-1}\chi_1(p)^j\chi_2(p)^{2r-j}\right]p^{r(-2s+1)}.
     \end{align*}
   Now the usual geometric series identity and \eqref{eq:generating_series} allows us to write  all of the above series as a rational expression in $\chi_1, \chi_2$ and a power of $p$. We obtain
   \begin{equation}\label{eq:rational_expression}
     \begin{split}
       \sum_{(k,l)\in \Lambda_+}\widehat{\chi}_{\mathscr{S}}(T_{k,l})p^{-s(k+l)} &= \frac{1+c_p\chi_1(p)\chi_2(p)p^{-2s+1} + (c_p-1)\chi_1(p)^2\chi_2(p)^2p^{-4s+2}}{(1-\chi_1(p)\chi_2(p)p^{-2s})(1-\chi_1(p)^2p^{-2s+1})(1-\chi_2(p)^2p^{-2s+1})} \\
       & = \frac{(1 + p\chi_1(p)\chi_2(p)p^{-2s})(1 + (c_p-1)p\chi_1(p)\chi_2(p)p^{-2s})}{(1-\chi_1(p)\chi_2(p)p^{-2s})(1-\chi_1(p)^2p^{-2s+1})(1-\chi_2(p)^2p^{-2s+1})}
  \end{split}
     \end{equation}
   By Lemma \ref{lem:explicit_satake_coprime} for $p$ coprime to $|D|$, $c_p$ is equal to $1-p^{-1}$  and the right-hand side of \eqref{eq:rational_expression} in this case becomes
   \begin{equation}\label{eq:rational_expression_coprime}
         \frac{1+p\chi_1(p)\chi_2(p)p^{-2s}}{(1-\chi_1(p)^2p^{-2s+1})(1-\chi_2(p)^2p^{-2s+1})}
     \end{equation}
   as the $(c_p-1)p = -1$ and the second factor in the numerator cancels. 
   For $p$ dividing $|D|$ we have $c_p=\gamma(D_p)(p-1)\delta_p$. Consequently, if $\ord_p(|D_p|)$ is odd, $\delta_p=0$ and the right-hand side of \eqref{eq:rational_expression} reads as
   \begin{equation}\label{eq:rational_expression_ncoprime_odd}
    \frac{1-\chi_1(p)^2\chi_2(p)^2p^{-4s+2}}{(1-\chi_1(p)\chi_2(p)p^{-2s})(1-\chi_1(p)^2p^{-2s+1})(1-\chi_2(p)^2p^{-2s+1})} .
     \end{equation}
   If $\ord_p(|D_p|)$ is even, $\delta_p = 1$ and we obtain for the right-hand side of \eqref{eq:rational_expression}
   \begin{equation}\label{eq:rational_expression_ncoprime_even}
    \frac{(1 + p\chi_1(p)\chi_2(p)p^{-2s})(1 + (\gamma_p(D_p)(p-1)-1)p\chi_1(p)\chi_2(p)p^{-2s})}{(1-\chi_1(p)\chi_2(p)p^{-2s})(1-\chi_1(p)^2p^{-2s+1})(1-\chi_2(p)^2p^{-2s+1})}
     \end{equation}
\end{proof}

Combining Lemma \ref{lem:int_repr_series} with Lemma \ref{lem:integral_series} immediately yields

\begin{theorem}\label{thm:formal_dirichlet_series_rational_expr}
  Let $\chi = (\chi_1,\chi_2)$ be an unramified character of $\calM_p$, 
  \[
  B_S(\chi, X)=
\sum_{(k,l)\in \Lambda_+}\widehat{\chi_S}(T_{k,l})X^{k+l} \text{ and }  B_\calS(\chi, X)=
\sum_{(k,l)\in \Lambda_+}\widehat{\chi_\calS}(T_{k,l})X^{k+l}. 
\]
Then $B_S(\chi,p^{-s})$ and $B_\calS(\chi,p^{-s})$ can be written as a rational expression in $\chi_1(p)$, $\chi_2(p)$:

If $p$ is coprime to $|D|$, 
\begin{equation}\label{eq:formal_dirichlet_series_rational_expr}
 B_S(\chi,p^{-s}) = \frac{1+p\chi_1(p)\chi_2(p)p^{-2s}}{(1-\chi_1(p)^2p^{-2s+1})(1-\chi_2(p)^2p^{-2s+1})}.
\end{equation}
If $p$ divides $|D|$, then  $B_\calS(\chi,p^{-s})$ is equal to 
\begin{equation}
  \frac{(1 + p\chi_1(p)\chi_2(p)p^{-2s})(1 + (\gamma_p(D_p)(p-1)-1)p\chi_1(p)\chi_2(p)p^{-2s})}{(1-\chi_1(p)\chi_2(p)p^{-2s})(1-\chi_1(p)^2p^{-2s+1})(1-\chi_2(p)^2p^{-2s+1})}
\end{equation}
for $\ord_p(|D|)\in 2\Z$ and
\begin{equation}
   \frac{1-\chi_1(p)^2\chi_2(p)^2p^{-4s+2}}{(1-\chi_1(p)\chi_2(p)p^{-2s})(1-\chi_1(p)^2p^{-2s+1})(1-\chi_2(p)^2p^{-2s+1})}
  \end{equation}
for $\ord_p(|D|)\in 2\Z+1$. 
\end{theorem}

Now let $F\in A_{\kappa}(\omega_f)$ be a common eigenform of all operators $\calT^{T_{k,l}}$, $(k,l)\in \Lambda_+$,  for all primes $p$ with eigenvalues $\lambda_{F,p}(T_{k,l})$.
Then according to \cite{St3}, Remark 5.10, 
  a $\C$-algebra homomorphism of $\calH^+(\calQ_p//\calK_p,\omega_p)$ (and $\calH(\calQ_p//\calK_p,\omega_p)$ for $(p,|D|)=1$)  is defined for each prime $p$ via the eigenvalues $\lambda_{F,p}$. By Lemma \ref{lem:algebra_hom}, $\lambda_{F,p}$ determines an unramified character $\chi_{F,p}=(\chi_{F,p}^{(1)},\chi_{F,p}^{(2)})$ of $\calM_p$ satisfying
\begin{equation}\label{eq:eigenvalue_char}
  \begin{split}
  \lambda_{F,p}(T_{k,l}) &=
  \begin{cases}
    \sum_{(r,s)\in \Lambda} S(\langle T_{k,l}, \varphi_p^{(0)}\rangle)(m(p^r,p^s))\chi_{F,p}(m(p^r,p^s)), &  (p, |D|)=1,\\
      \sum_{(r,s)\in \Lambda} \langle \calS(T_{k,l})(m(p^r,p^s)), \varphi_p^{(0)}\rangle\chi_{F,p}(m(p^r,p^s)), & p\mid |D|
  \end{cases}\\
  &=
  \begin{cases}
    \widehat{\chi_{F,p}}_S(T_{k,l}), & (p, |D|)=1,\\
    \widehat{\chi_{F,p}}_\calS(T_{k,l}), & p\mid |D|.
    \end{cases}
  \end{split}
\end{equation}

By virtue of Theorem \ref{thm:formal_dirichlet_series_rational_expr} we have
\begin{align*}
  \sum_{(k,l)\in \Lambda_+}\lambda_{F,p}(T_{k,l})p^{-s(k+l)} &= \sum_{(k,l)\in \Lambda_+} \widehat{\chi_{F,p}}_S(T_{k,l})p^{-s(k+l)} \\
  &=\frac{1+p\chi_{F,p}^{(1)}(p)\chi_{F,p}^{(2)}(p)p^{-2s}}{(1-\chi_{F,p}^{(1)}(p)^2p^{-2s+1})(1-\chi_{F,p}^{(2)}(p)^2p^{-2s+1})}\\
\end{align*}
if $(p,|D|)=1$,
\begin{align*}
  &\sum_{(k,l)\in \Lambda_+}\lambda_{F,p}(T_{k,l})p^{-s(k+l)} = \sum_{(k,l)\in \Lambda_+} \widehat{\chi_{F,p}}_\calS(T_{k,l})p^{-s(k+l)} \\
  &=\frac{(1+p\chi_{F,p}^{(1)}(p)\chi_{F,p}^{(2)}(p)p^{-2s})(1 +(\gamma(D_p)(p-1)-1)p\chi_{F,p}^{(1)}(p)\chi_{F,p}^{(2)}(p)p^{-2s})}{(1-\chi_{F,p}^{(1)}(p)\chi_{F,p}^{(2)}(p)p^{-2s})(1-\chi_{F,p}^{(1)}(p)^2p^{-2s+1})(1-\chi_{F,p}^{(2)}(p)^2p^{-2s+1})}\\
\end{align*}
if $p\mid |D|$ and $\ord_p(|D|)\in 2\Z$ and
\begin{align*}
  &\sum_{(k,l)\in \Lambda_+}\lambda_{F,p}(T_{k,l})p^{-s(k+l)} = \sum_{(k,l)\in \Lambda_+} \widehat{\chi_{F,p}}_\calS(T_{k,l})p^{-s(k+l)} \\
  &=\frac{1-\chi_{F,p}^{(1)}(p)^2\chi_{F,p}^{(2)}(p)^2p^{-4s+2}}{(1-\chi_{F,p}^{(1)}(p)\chi_{F,p}^{(2)}(p)p^{-2s})(1-\chi_{F,p}^{(1)}(p)^2p^{-2s+1})(1-\chi_{F,p}^{(2)}(p)^2p^{-2s+1})}
\end{align*}
if $p\mid |D|$ and $\ord_p(|D|)\in 2\Z+1$

Following the classical literature (e.g.  \cite{BM}, Chap. 7.2,  and \cite{BoSP}, Chap. I.$\S$2), these identities give rise to the definition of a {\it standard $L$-function} associated to $F$.

\begin{definition}\label{def:standard_L_func}
  Let $F\in A_\kappa(\omega_f)$ be a common eigenform of all operators $\calT^{T_p}$ for any prime $p$ and any $T_p\in \calH^+(\calQ_p//\calK_p,\omega_p)$ (and $\calH(\calQ_p//\calK_p,\omega_p)$). We define the standard $L$-function of $F$ by
  \begin{equation}
    L(s,F) = \prod_{p<\infty}L_p(s,F)
  \end{equation}
  with $L_p(s,F)$ equal to
  \begin{equation}
    \begin{split}
   & \frac{1+p\chi_{F,p}^{(1)}(p)\chi_{F,p}^{(2)}(p)p^{-2s}}{(1-\chi_{F,p}^{(1)}(p)^2p^{-2s+1})(1-\chi_{F,p}^{(2)}(p)^2p^{-2s+1})},\quad (p,|D|)=1, \\
      & \frac{(1+p\chi_{F,p}^{(1)}(p)\chi_{F,p}^{(2)}(p)p^{-2s})(1+ (\gamma(D_p)(p-1)-1)p\chi_{F,p}^{(1)}(p)\chi_{F,p}^{(2)}(p)p^{-2s})}{(1-\chi_{F,p}^{(1)}(p)\chi_{F,p}^{(2)}(p)p^{-2s})(1-\chi_{F,p}^{(1)}(p)^2p^{-2s+1})(1-\chi_{F,p}^{(2)}(p)^2p^{-2s+1})},\\
      &\quad p\mid |D| \text{ and } \ord_p(|D|)\in 2\Z,\\
      & \frac{1-\chi_{F,p}^{(1)}(p)^2\chi_{F,p}^{(2)}(p)^2p^{-4s+2}}{(1-\chi_{F,p}^{(1)}(p)\chi_{F,p}^{(2)}(p)p^{-2s})(1-\chi_{F,p}^{(1)}(p)^2p^{-2s+1})(1-\chi_{F,p}^{(2)}(p)^2p^{-2s+1})},\\
      &\quad p\mid |D| \text{ and } \ord_p(|D|)\in 2\Z+1.
   \end{split}
  \end{equation}
  Let $f\in S_\kappa(\rho_L)$ and $F_f$ be the  associated automorphic form. Based on \cite{St3}, Remark 5.10, we then define the standard $L$-function of $f$ by
  \begin{equation}
    L(s,f) = L(s,F). 
  \end{equation}
  \end{definition}

\subsection{Analytic properties of $L(s,F)$}
In this section we study the analytic properties of the $L$-function $L(s, F)$. It turns out that it can be continued meromorphically to the whole $s$-plane.

\subsubsection{Relation to the standard zeta-function $\calZ(s,f)$}\label{subsec:rel_standard_zeta}
The subsequent exposition is essentially due to Arakawa, \cite{Ar}, Theorem 5.5, tailored to our setting. 

Let $T_{k,l}$ be the operator in \cite{St3}, Corollary 4.7. or Theorem 4.11, and  $F\in A_\kappa(\omega_f)$ a common eigenform of all operators $\calT^{T_{k,l}}$ with eigenvalues $\lambda_{F,p}(T_{k,l})$ and $f_F$ the corresponding modular form in $S_\kappa(\rho_L)$.
Then
\[
\sum_{(k,l)\in \Lambda_+}\calT^{T_{k,l}}(F) = \left(\sum_{(k,l)\in \Lambda_+}\lambda_{F,p}(T_{k,l})p^{-s(k+l)}\right)F
\]
is a well defined element in $A_\kappa(\omega_f)$ given $\re(s)$ is sufficiently large.
 
\begin{align*}
  \mathscr{A}^{-1}\left(\sum_{(k,l)\in \Lambda_+}p^{-s(k+l)}\calT^{T_{k,l}}(F)\right) &= \sum_{(k,l)\in \Lambda_+}p^{-s(k+l)}\mathscr{A}^{-1}\left(\calT^{T_{k,l}}(F)\right)\\
  &=\sum_{(k,l)\in \Lambda_+}p^{-s(k+l)}f_{\calT^{T_{k,l}}(F)}.
\end{align*}
Now, because of  \eqref{eq:eigenvalue_char}
\begin{align*}
  \sum_{(k,l)\in \Lambda_+}p^{-s(k+l)}\calT^{T_{k,l}}(F) &= \left(\sum_{(k,l)\in \Lambda_+}p^{-s(k+l)}\lambda_{F,p}(T_{k,l})\right)F
  \\ 
  &= \begin{cases} 
    \left(\sum_{(k,l)\in \Lambda_+}\widehat{\chi_{F,p}}_S(T_{k,l})p^{-s(k+l)}\right)F, & (p,|D|)=1\\
    \left(\sum_{(k,l)\in \Lambda_+}\widehat{\chi_{F,p}}_\calS(T_{k,l})p^{-s(k+l)}\right)F, & p\mid |D|
  \end{cases} \\
  &= \begin{cases}
    B_S(\chi_{F,p},p^{-s})F, & (p,|D|)=1\\
    B_\calS(\chi_{F,p}, p^{-s})F, & p\mid |D|.
    \end{cases}
\end{align*}
On the other hand, using \cite{St3}, (5.28),  we find
\begin{align*}
  \sum_{(k,l)\in \Lambda_+}p^{-s(k+l)}f_{\calT^{T_{k,l}}(F)} &= \sum_{(k,l)\in \lambda_+}p^{-(s-\kappa/2 +1)(k+l)}T(m(p^{-k},p^{-l}))(f_F)\\
  & =\calZ_p(s-\kappa/2+1,f_F)f_F
\end{align*}
since $f_F$ is an eigenform of all Hecke operators $T(m(p^{-k},p^{-l}))$ by Remark 5.10 in \cite{St3}.
The calculations before show
\begin{align*}
 \sum_{(k,l)\in \Lambda_+}p^{-s(k+l)}f_{\calT^{T_{k,l}}(F)} &= \mathscr{A}^{-1}\left(\sum_{(k,l)\in \Lambda_+}p^{-s(k+l)}\calT^{T_{k,l}}(F)\right) 
  \\
&=  \begin{cases}
  B_S(\chi_{F,p},p^{-s})f_F, & (p,|D|)=1\\
    B_\calS(\chi_{F,p}, p^{-s})f_F, & p\mid |D|.
  \end{cases}
   \end{align*}
It follows
\begin{align*}
  \calZ_p(s-\kappa/2+1,f_F) &=
 \begin{cases}
  B_S(\chi_{F,p},p^{-s}), & (p,|D|)=1\\
  B_\calS(\chi_{F,p}, p^{-s}), & p\mid |D|.
  \end{cases}
\end{align*}
With the help of Theorem \ref{thm:formal_dirichlet_series_rational_expr} we finally obtain

\begin{theorem}\label{thm:rel_zeta_standard_L_func}
  Let  $D$ be an anisotropic discriminant form, $f\in S_\kappa(\rho_L)$ be a common eigenform of all Hecke operators $T(m(p^{-k}, p^{-l})), \; (k,l)\in \Lambda,$ and $F_f\in A_\kappa(\omega_f)$ the corresponding automorphic form. Then
  \begin{equation}\label{eq:rel_zeta_standard_L_func}
    \calZ(s-\kappa/2+1, f) = L(s,F_f). 
    \end{equation}
  \end{theorem}

\begin{theorem}\label{them:analytic_properties_L_func}
  Let $\kappa\in 2\Z,\;\kappa\ge 3$, satisfy $2\kappa +\sig(L)\equiv 0\bmod{4}$,  $D$ be an anisotropic discriminant form and $f\in S_\kappa(\rho_L)$ a common eigenform of all Hecke operators $T(m(k^2,1),\; k\in \N)$. Then the standard $L$-function $L(s,f)$ can be meromorphically continued to the whole $s$-plane.
\end{theorem}
\begin{proof}
  The proof merely boils down to use Theorem \ref{thm:analytic_propo_zeta_func} to establish the fact that $\calZ(s,f)$ is meromorphic on the whole $s$-plane and subsequently Theorem \ref{thm:rel_zeta_standard_L_func} to transfer the analytic properties of $\calZ(s-\kappa/2+1, f)$ to $L(s, F_f)$. 
  \end{proof}

\end{document}